\documentclass[11pt]{article}

\usepackage{authblk}
\usepackage{amsmath,amssymb,amsthm}
\usepackage{mathrsfs}
\usepackage{graphicx,epsfig,psfrag}
\usepackage{xcolor}
\usepackage{enumerate}
\usepackage{subcaption}

\allowdisplaybreaks

\newtheorem{theorem}		{Theorem}[section]
\newtheorem{corollary}		[theorem]{Corollary}
\newtheorem{proposition}		[theorem]{Proposition}
\newtheorem{definition}		[theorem]{Definition}
\newtheorem{lemma}		[theorem]{Lemma}
\newtheorem*{lemma*}{Lemma}
\newtheorem{sideremark}		[theorem]{Remark}
\newtheorem{sidenote}		[theorem]{Note}
\newtheorem{sideeg}		[theorem]{Example}
\newtheorem{sideconj}		[theorem]{Conjecture}
\newtheorem{sideassumption}	{Assumption}

\newenvironment{remark}		{\begin{sideremark}\rm}{\end{sideremark}}

\DeclareMathOperator*{\stat}{stat}
\DeclareMathOperator*{\argstat}{arg\,stat}

\newcommand{\ba}				{\begin{array}}
\newcommand{\ea}				{\end{array}}

\newcommand{\er}[1]			{{(\ref{#1})}}

\newcommand{\ddtone}[2]			{{\frac{d {#1}}{d {#2}}}}
\newcommand{\pdtone}[2]			{{\frac{\partial {#1}}{\partial {#2}}}}
\newcommand{\ol}[1]				{{\overline{#1}}}
\newcommand{\op}[1]			{{\mathcal{#1}}}
\newcommand{\ts}[1]				{{\textstyle{#1}}}
\newcommand{\wh}[1]			{{\widehat{#1}}}
\newcommand{\wt}[1]			{{\widetilde{#1}}}

\newcommand{\Frechet}			{{Fr\'{e}chet}}

\newcommand{\demi}			{{\ts{\frac{1}{2}}}}
\newcommand{\nn}				{{\nonumber}}

\newcommand{\bo}				{{\mathscr{L}}}
\newcommand{\cB}				{{\mathscr{B}}}
\newcommand{\cU}				{{\mathscr{U}}}
\newcommand{\cX}				{{\mathscr{X}}}
\newcommand{\cY}				{{\mathscr{Y}}}
\newcommand{\cZ}				{{\mathscr{Z}}}
\newcommand{\grad}			{{\nabla}}

\newcommand{\N}				{{\mathbb{N}}}
\newcommand{\R}				{{\mathbb{R}}}
\newcommand{\Z}				{{\mathbb{Z}}}

\title{Exploiting characteristics in stationary action problems\thanks{Partially supported by AFOSR grants FA2386-16-1-4066. A preliminary version of this work appeared in \cite{DMY1:19}.}
}
\author[1]{Vincenzo Basco \thanks{\texttt{vincenzo.basco@unimelb.edu.au, vincenzobasco@gmail.com}}}
\author[2]{Peter M. Dower\thanks{\texttt{pdower@unimelb.edu.au}}}
\author[3]{William M. McEneaney\thanks{\texttt{wmceneaney@ucsd.edu}}}
\author[4]{Ivan Yegorov\thanks{\texttt{ivan.egorov@ndus.edu,  ivanyegorov@gmail.com}}}
\affil[1,2]{\small Dept. Electrical \& Electronic Eng., University of Melbourne, Victoria 3010, Australia.}
\affil[3]{Dept. Mech. \& Aero. Eng., UC San Diego, La Jolla, CA 92093 USA.}
\affil[4]{Dept. of Mathematics, North Dakota State University, Fargo ND 58102, USA.}
\date{}

\begin{document}



\maketitle








\begin{abstract} 
\noindent
Connections between the principle of stationary action and optimal control, and between established notions of minimax and viscosity solutions, are combined to describe trajectories of energy conserving systems as solutions of corresponding Cauchy problems defined with respect to attendant systems of characteristics equations. 
\end{abstract}



\


\section{Introduction}
\label{sec:intro}

In recent investigations \cite{MD1:15,DM1:17}, connections between Hamilton's action principle and optimal control have been exploited to synthesize fundamental solutions for conservative systems of differential equations, in finite and infinite dimensions, and their related two point boundary value problems (TPBVPs). In each case, an optimal control problem is identified whose cost is representative of the desired {\em action}, leading to a characteristic system corresponding to the desired conservative system. The tools of optimal control, including dynamic programming, semigroup theory, idempotent algebra, and convex analysis, subsequently provide a pathway for construction of its fundamental solution, for large classes of boundary conditions, see for example \cite{MD1:15}.

For short time horizons, convexity of the action functional with respect to the momentum trajectory is typically guaranteed (for finite dimensional dynamics). This ensures that an associated optimal control problem is well-defined, c.f. \cite{MD1:15}. Consequently, stationary action is achieved as {\em least} action, as characterised by a corresponding value function, while the associated equations of motion are described by the characteristic system corresponding to a standard Hamilton-Jacobi-Bellman partial differential equation. 

For longer or infinite time horizons, or for configurations with infinite dimensional dynamics, the equivalence of stationary action and optimal control breaks down, typically due to a loss of convexity of the action or to a presence of state constraints, in which further controllability assumptions are needed \cite{bascofrankcann2017necessary,bascofrankowska2018hjbe,bascofrankowska2018lipschitz}. This leads to finite escape phenomena exhibited by the value function, and hence an inability to propagate solutions beyond these times. This limitation is particularly severe in the infinite dimensional setting \cite{DM1:17}, and motivates exploration of {\em stationary} control problems, as opposed to optimal control problems, whose value can propagate through these finite escape phenomena to longer horizons \cite{DM1:17,MD1:17,MD1:18}.

An optimal control problem can be relaxed to a stationary control problem by formally replacing the infimum (or supremum) operation in the definition of the attendant value function with a {\em stat} operation \cite{MD1:17,MD1:18}. As indicated, this stat operation requires only stationarity of its cost function argument, rather than optimality. In the stationary action problems considered to date, see for example \cite{MD1:17,DM1:17}, this has involved the characterization of open loop controls that render the cost stationary. However, motivated by the notion of minimax solutions considered in \cite{S:95,YD1:18}, it is also reasonable to consider initial adjoint or momentum variables that render an associated characteristics based cost stationary. An investigation in this direction forms the basis of this paper, building on the preliminary work of \cite{DMY1:19}. The main results document an equivalence between two stationary control problems, subject to uniqueness of solutions of an attendant TPBVP, and a verification result for stationary trajectories. An illustrative example is included.

In terms of organization, in Section \ref{preliminaries} we recall some basic definitions and state the main assumptions of the present paper. The connection between least action and optimal control is reviewed in Section \ref{sec:optimal}, along with the relevant notion of minimax solution \cite{S:95,YD1:18}. Optimality in the attendant minimax solution definition is then relaxed to stationarity in Section \ref{sec:stationary}, and its relationship to the earlier work \cite{DM1:16} established. The paper concludes with a simple example in Section \ref{sec:example} and an Appendix.



\section{Preliminaries and main assumptions}\label{preliminaries}

Throughout, $\R$, $\Z$, $\N$ denote the real, integer, and natural numbers respectively, with extended reals defined as $\ol{\R} \doteq \R\cup\{\pm\infty\}$. $|\cdot|$ and $\langle\cdot,\cdot \rangle$ stands for the Eucliden and the standard scalar product, respectively. The space of continuous mappings between Banach spaces $\cX$ and $\cY$ is denoted by $C(\cX;\cY)$. The set of bounded linear operators between the two spaces is denoted by $\bo(\cX;\cY)$, or $\bo(\cX)$ if $\cX$ and $\cY$ coincide. Let $I\subset \mathbb R$ a closed interval. If $\cX$ is a real Hilbert space, we denote by $ L^2(I;\cX)$ the space of square summable measurable functions on $I$ endowed with the standard inner product. 

	Let $\cX,\,\cY$ two Banach spaces. A function $f\in C(\cX;\cY)$ is said to be \textit{{\Frechet} differentiable at $x\in\cX$}, with derivative $Df(x)\in\bo(\cX;\cY)$, if $ \lim_{\|h\|_\cX\rightarrow 0} \|df_x(h)\|_\cY=0$
	in which $df_x:\cX\rightarrow\cX$ is defined by
	\begin{align}
	\label{eq:df}
	& df_x(h)
	\doteq \left\{ \ba{cl}
	0 & \|h\|_\cX = 0
	\\
	\frac{f(x+h) - f(x) - Df(x)\, h}{\|h\|_\cX}
	& \|h\|_\cX > 0.
	\ea
	\right.
	\end{align}
	By definition, the map $h\mapsto df_x(h)$ is continuous at $0$. The function $f$ is said to be \textit{{\Frechet} differentiable} if it is {\Frechet} differentiable at any $x\in \cX$.

A function $f$ is continuously {\Frechet} differentiable, written $f\in C^1(\cX;\cY)$, if $Df$ is continuous on $\cX$. Higher order {\Frechet} derivatives are similarly defined, with $f\in C^k(\cX;\cY)$ if $f$ is $k$-times {\Frechet} differentiable.

Let $\cX$ denote a real Hilbert space of instantaneous generalized positions, and let $T\in\R_{\ge 0}$ and $t\in[0,T]$ denote the final and initial times of the desired motion. Denote the corresponding real Hilbert space of generalized momentum trajectories by $\cU[t,T]\doteq L^2([t,T];\cX)$.

Given an initial generalized position $x\in\cX$, a potential field $V\in C^4(\cX;\R)$, and a coercive self-adjoint inertia operator $\op{M}\in\bo(\cX)$, the {\em action} is defined as an integrated Lagrangian encapsulated by a cost function $J_T(t,x,\cdot):\cU[t,T]\rightarrow\R$ using an artificial terminal cost $\psi\in C^4(\cX;\R)$. In particular, 
\begin{align}
J_T(t,x,u) & \doteq \int_t^T \demi\, \langle u_s,\, \op{M}\, u_s \rangle - V(\bar x_s) \, ds + \psi(\bar x_T),
\label{eq:cost}
\end{align}
in which $s\mapsto \bar x_s\in C([t,T];\cX)$ is the generalized position trajectory satisfying
\begin{align}
\label{eq:dynamics}
\bar x_s
& \doteq x + \int_t^s u_\sigma\, d\sigma,
\quad
s\in[t,T],
\end{align}
for all $u\in\cU[t,T]$.

It is assumed throughout that there exist constants $m\in\R_{>0}$ and $K\in\R_{\ge 0}$ such that for all $x,h\in \cX$
\begin{align}
\begin{cases}
m \, \|h\|^2 - \langle h, \, \op{M}\, h \rangle \le 0,\\
\max_{j=0,1,2}  \|D^j \grad^2 V(x)\|_{\bo(\cX)} \le \ts{\frac{K}{2}},\\
\|\grad^2 \psi(x)\|_{\bo(\cX)} \le \ts{\frac{K}{2}},
\label{eq:ass-M-V}
\end{cases}
\end{align}
i.e., $\op{M}$ is coercive (and hence boundedly invertible), while second and third derivatives of the potential field and second derivative of the terminal cost are uniformly bounded.

\section{Least action and optimal control}
\label{sec:optimal}


For sufficiently short time horizons, an optimal control problem can be formulated that encapsulates stationary action as {\em least} action. Given $T\in\R_{\ge 0}$ sufficiently small, this optimal control problem is defined via a value function $W_T:[0,T]\times\cX\rightarrow\ol{\R}$ given by
\begin{align}
	W_T(t,x)
	& \doteq \inf_{u\in\cU[t,T]} J_T(t,x,u)
	\label{eq:optimal-value}
\end{align}
for all $t\in[0,T]$, $x\in\cX$. 

\begin{theorem}
\label{thm:optimal-value}
Given $T\in\R_{\ge 0}$, $t\in[0,T]$ such that $\max(T-t,1) \, (T-t) < \frac{m}{K}$, c.f. \er{eq:ass-M-V}, the value function $W_T(t,\cdot)$ of \er{eq:optimal-value} is well defined and real valued.
\end{theorem}

The proof of Theorem \ref{thm:optimal-value} uses the following lemma.

\begin{lemma}
\label{lem:cost-pre-coercive}
Given arbitrary $T\in\R_{\ge 0}$, $t\in[0,T]$, $x\in\cX$, and $u\in\cU[t,T]$, cost $J_T(t,x,\cdot):\cU[t,T]\rightarrow\R$ is twice {\Frechet} differentiable, with second derivative $D_u\grad_u J_T(t,x,u)\in\bo(\cU[t,T])$ given by
\begin{align}
	\label{eq:diff-J}
	& D_u \grad_u J_T(t,x,u)\, \delta_u
	= (\op{M} - \Delta_T(t,x,u))\, \delta_u
\end{align}
for all $\delta_u\in\cU[t,T]$, in which $\grad_u J_T(t,x,u)\in\cU[t,T]$ denotes the Riesz representation of the first {\Frechet} derivative at $u$, with $\Delta_T(t,x,u)\in\bo(\cU[t,T])$ given by
\begin{align}
	[\Delta_T(t,x,u)\, \delta_u]_r
	& \doteq \int_t^T \left[ \int_{r\vee\rho}^T \grad^2 V(\bar x_\sigma) \, d\sigma - \grad^2 \psi(\bar x_T) \right] [\delta_u]_\rho\  d\rho
	\nn
\end{align}
for all $r\in[t,T]$, $\delta_u\in\cU[t,T]$.

Moreover, 
\begin{align}
	& \langle \delta_u, \, D_u \grad_u J_T(t,x,u)\, \delta_u \rangle_{\cU[t,T]}
	\ge K\, (\ts{\frac{m}{K}} - \max (T-t,1)\, (T-t))\, \|  \delta_u \|_{\cU[t,T]}^2
	\label{eq:pre-coercive}
\end{align}
for all $\delta_u\in\cU[t,T]$, so that $D_u\grad_u J_T(t,x,u)$ is coercive and $J_T(t,x,\cdot)$ is strictly convex and proper, provided that $T-t\in\R_{\ge 0}$ is sufficiently small.
\end{lemma}
\begin{proof}
Fix $T\in\R_{\ge 0}$, $t\in[0,T]$, $x\in\cX$, and $u\in\cU[t,T]$. H\"older's inequality and the second inequality in \er{eq:ass-M-V} yield $\Delta_T(t,x,u)\in\bo(\cU[t,T])$, with 
\begin{align}
	& \| \Delta_T(t,x,u) \|_{\bo(\cX)} \le K \max(T-t,1)\, (T-t).
	\label{eq:Delta-bound}
\end{align}
Twice {\Frechet} differentiability of $J_T(t,x,\cdot):\cU[t,T]\rightarrow\R$, boundedness of the second derivative, i.e. $D_u\grad_u J_T(t,x,u)\in\bo(\cU[t,T])$, and \er{eq:diff-J}, subsequently follow by a minor generalization of \cite[Theorem 3.6]{DM1:16}. Combining \er{eq:Delta-bound} with the first inequality in \er{eq:ass-M-V} via Cauchy-Schwartz and \er{eq:diff-J} yields
\begin{align}
	\langle \delta_u, \, D_u \grad_u J_T(t,x,u)\, \delta_u \rangle_{\cU[t,T]}
	& = \langle \delta_u, \, (\op{M} - \Delta_T(t,x,u))\, \delta_u \rangle_{\cU[t,T]}
	\nn\\
	& \ge \langle \delta_u, (m - K \max(T-t,1)\, (T-t))\, \delta_u \rangle_{\cU[t,T]}
	\nn\\
	& = K\, (\ts{\frac{m}{K}} - \max(T-t,1)\, (T-t))\, \|  \delta_u \|_{\cU[t,T]}^2,
	\nn
\end{align}
for all $\delta_u\in\cU[t,T]$, which is \er{eq:pre-coercive}. By inspection, for sufficiently short time horizons, i.e. 
$
	\max(T-t,1)\, (T-t) < \ts{\frac{m}{K}}
$,
it follows that $D_u \grad_u J_T(t,x,u)\in\bo(\cU[s,t])$ is coercive, so that $J_T(t,x,\cdot)$ is strictly convex. Given an arbitrary $\hat u\in\cU[t,T]$, and $\tilde u \doteq \hat u - u\in\cU[t,T]$, Taylor's theorem further implies that
\begin{align}
	J_T(t,x,\hat u)
	& = J_T(t,x,u) + \langle \tilde u, \grad_u J_T(t,x,u) \rangle_{\cU[t,T]}
	\nn\\
	& \qquad + \left\langle \tilde u, \, \left( \int_0^1 (1-\sigma)\, D_u \grad_u J_T(t,x,u + \sigma\, \tilde u) \, d\sigma \right) \tilde u
	\right\rangle_{\cU[t,T]}
	\nn\\
	& \ge J_T(t,x,u) + \langle \tilde u, \grad_u J_T(t,x,u) \rangle_{\cU[t,T]}
	+ \demi\, K\, (\ts{\frac{m}{K}} - \max(T-t,1)\, (T-t))\, \|  \tilde u \|_{\cU[t,T]}^2.
	\nn
\end{align}
Hence, $\lim_{\|\hat u\|_{\cU[t,T]} \rightarrow \infty} J_T(t,x,\hat u) = \infty$, and
$J_T(t,x,\cdot)$ is proper.
\end{proof}

\begin{proof}[{Proof of Theorem \ref{thm:optimal-value}}]
Fix $T\in\R_{\ge 0}$, $t\in[0,T]$ as per the theorem statement, and any $x\in\cX$. Lemma \ref{lem:cost-pre-coercive} implies that the cost $J_T(t,x,\cdot):\cU[t,T]\rightarrow\R$ of \er{eq:cost} is strictly convex and proper. Hence, the value function \er{eq:optimal-value} is well defined, with the infimum achieved via a minimum, thereby yielding a real valued optimal cost.
\end{proof}

\begin{remark}
	Theorem \ref{thm:optimal-value} ensures that, for sufficiently short time horizons, the principle of stationary action can be formulated as a {\em least} action principle, via the optimal control problem defined by the value function \er{eq:optimal-value}. Applying standard tools from optimal control \cite{BCD:97,CS:04}, this value function may subsequently be characterized via the viscosity solution of a Hamilton-Jacobi-Bellman (HJB) partial differential equation (PDE). 
	\hfill{$\square$}
\end{remark}

\begin{remark}
	With $\cX$ finite dimensional, $V$ and $\psi$ bounded, and $T\in\R_{\ge 0}$, 
	\cite[Theorems 5.2.12, 7.4.14]{CS:04} implies that
	the value function $W_T$ of \er{eq:optimal-value} is the unique viscosity solution of 
	\begin{gather}
	\label{eq:HJB}
	\left\{
	\begin{aligned}
	&- \pdtone{W_T}{t}(t,x) + H(x,\grad_x W_T(t,x))=0 &&(t,x)\in [0,T]\times \cX
	\\
	&W_T(T,x) = \psi(x) &&x\in \cX,
	\end{aligned}
	\right.
	\end{gather}
	in which the Hamiltonian $H:\cX\times\cX\rightarrow\R$ is given via completion of squares by
	\begin{align}
	H(x,p)
	& \doteq V(x) + \demi\, \langle p, \, \op{M}^{-1}\, p \rangle
	= V(x) + \sup_{u\in\cX}\left\{ - \langle p, \, u \rangle - \demi\, \langle u,\, \op{M}\, u \rangle \right\}
	\label{eq:H}
	\end{align}
	for all $x,p\in\cX$.	
	Alternatively, boundedness of $V$ and $\psi$ can be omitted for problems for which $u\in\cU$ where $\cU$ is a compact subset of $\cX$, see \cite[Theorem 7.4.14]{CS:04}.
	\hfill{$\square$}
\end{remark}


The characteristic system associated with \er{eq:HJB} is
\begin{align}
	\label{eq:char}
	\left\{ \begin{aligned}
	\dot {\bar x}_s
	& = -\grad_p H(\bar x_s, \bar p_s) = \bar u_s \doteq -\op{M}^{-1}\, \bar p_s,
	& \bar x_T & = y,
	\\
	\dot {\bar p}_s
	& = \grad_x H(\bar x_s, \bar p_s) = \grad V(\bar x_s),
	& \bar p_T & = \grad \psi(y),
	\\
	\dot {\bar z}_s
	& = -\langle \bar p_s, \grad_p H(\bar x_s, \bar p_s) \rangle + H(\bar x_s, \bar p_s) 
	&&
	\\
	& \qquad = V(\bar x_s) - \demi\, \langle \bar p_s, \op{M}^{-1}\, \bar p_s \rangle,
	& \bar z_T & = \psi(y),
	\end{aligned}
	\right.
\end{align}
for all $s\in[t,T]$, $y\in\cX$, in which $\grad_x$ and $\grad_p$ denote the Riesz representations of the respective {\Frechet} derivatives. The first two equations in \er{eq:char} correspond to the equations of motion defined by the action principle. As formalised later in Lemma \ref{lem:X-classical}, these equations coupled with either initial or terminal data exhibit a unique classical solution. In particular, the second derivative $\ddot {\bar x}_s$ is well defined, with
\begin{align}
	& \ddot {\bar x}_s = -\op{M}^{-1} \, \dot {\bar p}_s = - \op{M}^{-1}\, \grad V(\bar x_s)
	\nn
\end{align}
for all $s\in[t,T]$, which is Newton's second law. Observe also that the Hamiltonian $H$ of \er{eq:H} corresponds to the total energy, i.e. the sum of potential and kinetic energies. As expected, the chain rule implies that
\begin{align}
	\ts{\ddtone{}{s}} H(\bar x_s, \bar p_s)
	& = \langle \grad_x H(\bar x_s, \bar p_s), \dot {\bar x}_s \rangle 
				+ \langle \grad_p H(\bar x_s, \bar p_s), \dot {\bar p}_s \rangle
	\nn\\
	& = -\langle \grad_x H(\bar x_s, \bar p_s), \grad_p H(\bar x_s, \bar p_s) \rangle 
	+ \langle \grad_p H(\bar x_s, \bar p_s), \grad_x H(\bar x_s, \bar p_s) \rangle
	\nn\\
	& = 0,
	\nn
\end{align}
for all $s\in[t,T]$, i.e. energy is conserved.


\section{Stationary action}
\label{sec:stationary}
The connection between least action and optimal control is known to break down for longer time horizons, due typically to a loss of convexity of the action encapsulated by the cost \er{eq:cost}. This can be seen in Lemma \ref{lem:cost-pre-coercive}, where the convexity guarantee provided for the cost \er{eq:cost} is no longer valid, thereby rendering the optimal control interpretation of Theorem \ref{thm:optimal-value} inapplicable. In practice, the value function \er{eq:optimal-value} involved experiences finite escape phenomena as the horizon increases and convexity of the cost is lost. 

On longer time horizons, it is well known that the stationary (rather than least) action principle continues to describe the motion of energy conserving systems. In order to encapsulate this description in a framework that is analogous to optimal control, the infimum operation appearing in \er{eq:optimal-value} is relaxed to a {\em stat} operation \cite{MD1:17,MD1:18}.

\begin{definition}\rm
	The {\em stat} operation, along with the corresponding {\em argstat} operation, with respect to a function $F\in C^1(\cZ;\R)$ is defined by
	\begin{align}
	\stat_{\zeta\in\cZ} F(\zeta)
	\doteq \left\{ F(\bar\zeta) \, \biggl| \, \bar\zeta\in\argstat_{\zeta\in\cZ} F(\zeta) \right\},
	\quad
	\argstat_{\zeta\in\cX} F(\zeta)
	\doteq \left\{ \zeta\in\cZ \, \biggl| \, 0 = \lim_{y\rightarrow\zeta} \frac{|F(y) - F(\zeta)|}{\|y - \zeta\|} \right\},
	\label{eq:stat}	
	\end{align}
	in which $\cZ$ is a Banach space. The elements in $\argstat_{\zeta\in\cX} F(\zeta)$ are called \textit{stationary points} for $F$.
\end{definition}
 Relaxing the infimum in \er{eq:optimal-value} to {\em stat} as indicated gives rise to the notion of a {\em stationary} control problem.


\subsection{Stationary control problems}
With $\cZ\doteq\cU[t,T]$, the relaxed (and possibly set-valued) {\em stationary} value function $\wt{W}_T$ of interest is defined by \cite{MD1:17,MD1:18,DM1:17,DM1:16}
\begin{align}
	\wt W_T(t,x)
	& \doteq \stat_{u\in\cU[t,T]} J_T(t,x,u),
	\label{eq:stat-u-value}
\end{align}
for all $t\in[0,T]$, $x\in\cX$, in which $J_T$ is the same cost \er{eq:cost}. The utility of \er{eq:stat-u-value}, relative to \er{eq:optimal-value}, in recovering the desired dynamics on arbitrary time horizons is illustrated via the following standard calculus of variations result.

\begin{theorem}
\label{thm:stat-u-char}
Given $T\in\R_{\ge 0}$, $t\in[0,T]$, $x\in\cX$, a velocity trajectory $\bar u\in\cU[t,T]$ renders the cost \er{eq:cost} stationary, i.e. $\bar u$ satisfies
\begin{gather}
	\label{eq:stat-u-char}
	\bar u\in\argstat_{u\in\cU[t,T]}  J_T(t,x,u),
\end{gather}
if and only if there exists a mild solution $(\bar x, \bar p)\in(\cU[t,T])^2$ of the TPBVP defined by \er{eq:char} with $\bar x_t = x$ fixed. Furthermore, $\bar u$ in \er{eq:stat-u-char} satisfies
\begin{align}
	\label{eq:stat-u-bar}
	\bar u_s & = -\op{M}^{-1}\, \bar p_s 
	\quad\text{a.e.}\quad s\in[t,T].
\end{align}
\end{theorem}
\begin{proof}
See for example \cite[Theorem 3.9]{DM1:16}.
\end{proof}

Rather than focus immediately on a dynamic programming style approach to the synthesis of stationary trajectories \cite{MD1:17}, or to solving the TPBVP highlighted in Theorem \ref{thm:stat-u-char}, the aim is instead to consider an alternative to cost \er{eq:cost} that appeals directly to the underlying characteristics system \er{eq:char}. With this in mind, first observe that integration of the final equation of \er{eq:char} yields
\begin{align}
	\label{eq:z-t}
	\bar z_t 
	& 
	= \psi(y) + \int_t^T \demi\, \langle \bar u_s, \op{M}\, \bar u_s \rangle - V(\bar x_s) \, ds,
\end{align}
which is of the same form as the cost $J_T(t,x,\bar u)$ of \er{eq:cost}, provided that $\bar x_t = x$. For short horizons, this motivates an equivalent characterization of the optimal control value function \er{eq:optimal-value} as the unique {\em minimax solution} (also called {\em minimal selection}) \cite{S:95} of \er{eq:HJB}, given by
\begin{align}
	\label{eq:minimax-value}
	& \left\{ \begin{aligned}
	& W_T(t,x)
	= \inf_{y\in Y_{T}(t,x)} \bar z_t,
	\\
	& Y_{T}(t,x)
	\doteq \{ y\in\cX \, | \, \bar x_t = x, \, \bar x_T = y \}
	\end{aligned}
	\right.
\end{align}
for all $t\in[0,T]$, $x\in\cX$. In view of \er{eq:z-t}, and as per \cite{YD1:18}, \er{eq:minimax-value} may be recast as an optimization over an initial adjoint variable, i.e. 
\begin{align}
	W_T(t,x)
	& \doteq \inf_{p\in\cX} \bar J_T(t,x,p)
	\label{eq:optimal-p-value}
\end{align}
for all $t\in[0,T]$, $x\in\cX$, in which $\bar J_T(t,x,\cdot):\cX\rightarrow\R$ is the associated cost defined by
\begin{align}
	\bar J_T(t,x,p)
	& \doteq \int_t^T \demi\, \langle \bar p_s, \, \op{M}^{-1}\, \bar p_s \rangle - V(\bar x_s)\, ds + \psi(\bar x_T),
	\label{eq:p-cost}
\end{align}
for all $t\in[0,T]$, with respect to the Cauchy problem
\begin{align}
	\left\{ \begin{aligned}
	\dot {\bar x}_s & = -\op{M}^{-1}\, \bar p_s,
	& \bar x_t & = x,
	\\
	\dot {\bar p}_s & = \grad V(\bar x_s),
	& \bar p_t & = p,
	\end{aligned} \right.
	\label{eq:Cauchy-dynamics}
\end{align}
for all $s\in[t,T]$, $x,p\in\cX$, extracted from \er{eq:char}.

In view of \er{eq:optimal-p-value}, \er{eq:p-cost}, \er{eq:Cauchy-dynamics}, an alternative value function to \er{eq:stat-u-value} may be proposed by relaxing the infimum operation in \er{eq:optimal-p-value} to the stat operation \er{eq:stat}. In particular, define
the value function $\ol{W}_T$ by
\begin{align}
	\ol{W}_T(t,x)
	& \doteq \stat_{p\in\cX} \bar J_T(t,x,p)
	\label{eq:stat-p-value}
\end{align}
for all $T\in\R_{\ge 0}$, $t\in[0,T]$, $x\in\cX$, with cost $\bar J_T$ as per \er{eq:p-cost}, and the stat involved possibly set-valued.

The subsequent analysis is concerned with \er{eq:stat-p-value}, and in particular existence of the associated argstat and its relationship to corresponding argstat in Theorem \ref{thm:stat-u-char}. With this analysis in mind, it is convenient for brevity of notation to define $f:\cX^2\rightarrow\cX^2$, $l:\cX^2\rightarrow\R$, and $\Psi:\cX^2\rightarrow\R$ by
\begin{align}
	\label{eq:f}
	f(X)
	& \doteq \left( \ba{c}
		-\op{M}^{-1}\, p
		\\ 
		\grad V(x)
	\ea \right),
	\quad
	X \doteq Y_p(x) \doteq \left( \ba{c}	
				x \\ p
			\ea \right)\in\cX^2,
	\\
	l(X)
	& \doteq 
	\demi\, \langle p, \, \op{M}^{-1}\, p \rangle - V(x),
	\quad
	\Psi(X) 
	\doteq \psi(x).
	\nn
\end{align}
Note that \er{eq:p-cost}, \er{eq:Cauchy-dynamics}, \er{eq:stat-p-value} correspond to
\begin{align}
	\label{eq:X-cost}
	& \wt J_T(t,Y)
	= \bar J_T(t,x,p) = \int_t^T l(X_s) \, ds + \Psi(X_T),
	\\
	\label{eq:X-dynamics}
	& \dot X_s = f(X_s), \ s\in[t,T],
	\ X_t = Y \doteq Y_p(x), 
	\\
	\label{eq:X-value}
	& \ol{W}_T(t,x)
	= \stat_{p\in\cX} \wt J_T\left(t, Y_p(x) \right). 
\end{align}


\subsection{{\Frechet} differentiation of the cost functional}
The objective now is to characterise the argstat in \er{eq:stat-p-value} via differentiation of \er{eq:X-cost} via \er{eq:p-cost}. With this in mind, applying classical arguments as those in \cite[Chapter 5]{P:83}, some intermediate lemmas are useful, whose proofs are referred to Appendix \ref{proofs_lemmata}.

\begin{lemma}
\label{lem:X-classical}
Given any $T\in\R_{\ge 0}$, $t\in[0,T]$, $Y\in\cX^2$, the initial value problem \er{eq:X-dynamics} has a unique classical solution $\ol{X}(Y)\in C([t,T];\cX^2)\cap C^1((t,T);\cX^2)$.
\end{lemma}


\begin{lemma}
\label{lem:X-of-Y-cts}
The map $Y\mapsto \ol{X}(Y)$ defined via the unique classical solution of Lemma \ref{lem:X-classical} is continuous, i.e. $\ol{X}\in C(\cX^2;C([t,T];\cX^2))$. In particular, there exists an $\alpha\in\R_{\ge 0}$ such that 
\begin{align}
	& \|\ol{X}(Y+h) - \ol{X}(Y)\|_\infty \le \|h\|\, \exp(\alpha\, (T-t))
	\label{eq:X-of-Y-cts}
\end{align}
for all $Y,h\in\cX^2$.
\end{lemma}

%

\begin{lemma}
\label{lem:X-of-Y-dble}
The map
$Y\mapsto\ol{X}(Y)\in C(\cX^2;C([t,T];\cX^2))$ of Lemma \ref{lem:X-of-Y-cts} is {\Frechet} differentiable with derivative given by
\begin{align}
	& D\ol{X}(Y)\in\bo(\cX^2;C([t,T];\cX^2)), 
	\qquad
	[D \ol{X}(Y)\, h]_s = U_{s,t}(Y) \, h,
	\quad s\in[t,T], 
	\label{eq:D-X-of-Y}
\end{align}
for all $Y,h\in\cX^2$, in which $U_{s,r}(Y)\in\bo(\cX^2)$, $r,s\in[t,T]$, is an element of the two-parameter family of evolution operators generated by $A_s\in\bo(\cX^2)$, i.e.
\begin{align}
	U_{s,r}\, h = U_{s,r}(Y)\, h
	& = h + \int_r^s A(Y)_\sigma\, U_{\sigma, r}(Y)\, h \, d\sigma,
	\qquad
	\forall \ r,s\in[t,T], \ h\in\cX,
	\label{eq:two-parameter}
\end{align}
in which $s\mapsto A(Y)_s$ is defined uniquely, given $Y$, by
\begin{align}
	& A(Y)_s \doteq \Lambda(\ol{X}(Y)_s) \doteq 
		\left( \ba{cc}
			0 & -\op{M}^{-1} \\
			\grad^2 V(\bar x(Y)_s) & 0
		\ea \right)\!,
	\label{eq:generator}
\end{align}
for all $s\in[t,T]$, and $\ol{X}(Y) = \left( \ba{c} \bar x(Y) \\ \bar p(Y) \ea \right)$.
\end{lemma}

\begin{lemma}
\label{lem:U-dble}
Given $T\in\R_{>0}$, $t\in[0,T)$, the map $Y\mapsto U_{s,r}(Y)$ is continuously {\Frechet} differentiable, uniformly in $r,s,\in[t,T]$.
\end{lemma}

Next, {\Frechet} regularity of the cost functional is demonstrated.

\begin{proposition}
\label{prop:X-cost-dble}
Given $T\in\R_{\ge 0}$, the map $(t,Y)\mapsto\wt J_T(t,Y)$ of \er{eq:X-cost} is continuously {\Frechet} differentiable with derivative $D \wt J_T$ given by
\begin{align}
\label{eq:D-J-of-Y}
D \wt J_T(t,Y)\, (\delta, h)
= \left( D_t \wt J_T(t,Y)\, \delta, \, D_Y \wt J_T(t,Y)\, h \right)
\end{align}
where
\begin{align}
\begin{cases}
D_t \wt J_T(t,Y)\, \delta= -l(\ol{X}(Y)_t) \, \delta,\\
D_Y \wt J_T(t,Y)\, h
= \langle \grad_Y \wt J_T(t,Y), h \rangle_{\cX^2},
\nn	
\end{cases}
\; \forall t\in[0,T),\,\forall Y,h\in\cX^2,\,\forall\delta\in(-t,T-t),
\end{align}
in which $\grad_Y \wt J_T(t,Y)\in\cX^2$ is the corresponding Riesz representation of $D_Y \wt J_T(t,Y)$, given by
\begin{align}
	\label{eq:X-grad-J}
	& \grad_Y \wt J_T(t,Y)
	= U_{T,t}(Y)' \, \grad \Psi(\ol{X}(Y)_T)+ \int_t^T U_{s,t}(Y)' \, \grad l(\ol{X}(Y)_s) \, ds.
\end{align}
Moreover, the map $(t,Y)\mapsto D \wt J_T(t,Y)$ is also continuously {\Frechet} differentiable.
\end{proposition}

\begin{proof}
	Fix $T\in\R_{\ge 0}$, $t\in[0,T]$, $Y,h\in\cX^2$, and $\delta\in(-t,T-t)$. Note immediately that $(t,Y)\mapsto \wt J_T(t,Y)$ is  {\Frechet} differentiable if and only if $t\mapsto \wt J_T(t,Y)$ and $Y\mapsto \wt J_T(t,Y)$ are {\Frechet} differentiable, using for example the norm $\|(t,Y)\|^2 \doteq |t|^2 + \|Y\|_{\cX^2}^2$. By inspection of \er{eq:X-cost}, $t\mapsto \wt J_T(t,Y)$ is {\Frechet} differentiable, with the derivative indicated in the left-hand equality in \er{eq:D-J-of-Y}. 
	
	In order to demonstrate that the map $Y\mapsto \wt J_T(t,Y)$ is {\Frechet} differentiable, with derivative as per the right-hand equality in \er{eq:D-J-of-Y}, the chain rule for {\Frechet} differentiation \cite{B:91} may be applied. To this end, in view of \er{eq:X-cost}, define $\wt I:C([t,T];\cX^2) \rightarrow \R$ and $\wt \iota:C([t,T];\cX^2)\rightarrow\bo(C([t,T];\cX^2);\R)$ by
	\begin{align}
	& \wt I(Z)
	\doteq \int_t^T l(Z_s) \, ds + \Psi(Z_T),
	\qquad
	\iota(Z)\, \delta
	\doteq \int_t^T \langle \grad l(Z_s), \delta_s \rangle \, ds + \langle \grad \Psi(Z_T),\delta_T \rangle_{\cX^2}
	\label{eq:pre-X-cost}
	\end{align}
	for all $Z,\delta\in C([t,T];\cX^2)$. Note in particular that $\wt J_T(t,Y) = \wt I\circ \ol{X}(Y)$, with $\ol{X}\in C(\cX^2; C([t,T];\cX^2)$  {\Frechet} differentiable by Lemma \ref{lem:X-of-Y-dble}, and the candidate derivative of $z\mapsto D\wt I(Z)$ is $\iota(Z)$ in \er{eq:pre-X-cost}. 
	Fix an arbitrary such $Z,\delta\in C([t,T];\cX^2)$. By inspection, 
	\begin{align}
	| \wt I(Z+\delta) - \wt I (Z) - \wt\iota(Z)\, \delta |
	& \le \int_t^T | l(Z_s + \delta_s) - l(Z_s) - \langle \grad l(Z_s), \delta_s \rangle_{\cX^2} | \, ds
	\nn\\
	& \qquad + | \Psi(Z_T + \delta_T) - \Psi(Z_T) - \langle \grad \Psi(Z_T),\delta_T \rangle_{\cX^2} |.
	\nn
	\end{align}
	As $l,\Psi\in C^3(\cX^2;\R)$ by \er{eq:ass-M-V}, \er{eq:f}, and $D l(Y) \, h = \langle \grad l(Y), h\rangle$, Taylor's theorem implies that
	\begin{align}	
	| \wt I(Z+\delta) - \wt I (Z) - \wt\iota(Z)\, \delta |
	& \le \int_t^T \left| \int_0^1 (1-\eta)\, \langle \delta_s, \grad^2 l(Z_s + \eta\, \delta_s) \, \delta_s \rangle_{\cX^2}\, d\eta \,
	\right| ds
	\nn\\
	& \qquad\qquad 
	+ \left| \int_0^1 (1-\eta)\, \langle \delta_T, \grad^2 \Psi(Z_T + \eta\, \delta_T)\, \delta_T \rangle_{\cX^2} \, d\eta \right|
	\nn\\
	& \le C \int_t^T \| \delta_s \|_{\cX^2}^2 \, ds + C \, \|\delta_T\|_{\cX^2}^2
	\le C\, \max(T-t,1)\, \|\delta\|_\infty^2
	\nn
	\end{align}
	in which $C<\infty$ is given by
	$$
	C \doteq \demi \sup_{Y\in\cX^2} \max( \|\grad^2 l(Y)\|_{\bo(\cX^2)}, \, \|\grad^2\Psi(Y)\|_{\bo(\cX^2)}),
	$$
	and finiteness follows by \er{eq:ass-M-V}, \er{eq:f}. Recalling \er{eq:df} yields $|d\wt I_Z(\delta)| \le C\, \max(T-t,1)\, \|\delta\|_\infty$, i.e. $\wt I$ is {\Frechet} differentiable with derivative $D\wt I = \wt \iota$.
	Hence, $\wt J_T(t,Y) = \wt I\circ \ol{X}(Y)$, in which $\wt I:C([t,T];\cX^2)\rightarrow C([t,T];\cX^2)$ is {\Frechet} differentiable, as demonstrated above, and $\ol{X}\in C(\cX^2; C([t,T];\cX^2)$ is {\Frechet} differentiable by Lemma \ref{lem:X-of-Y-dble}. The chain rule, along with \er{eq:D-X-of-Y}, \er{eq:pre-X-cost}, thus yield
	\begin{align}
	D_Y \wt J_T(t,Y)\, h
	& = D \wt I(\ol{X}(Y))\, D\ol{X}(Y)\, h 
	= \wt\iota(\ol{X}(Y))\, U_{\cdot,t}(Y)\, h
	\nn\\
	& = \int_t^T \langle \grad l(\ol{X}(Y)_s), U_{s,t}(Y)\, h \rangle \, ds 
	+ \langle \grad \Psi(\ol{X}(Y)_T),U_{T,t}(Y)\, h \rangle_{\cX^2}
	\nn\\
	& = \left\langle \int_t^T U_{s,t}(Y)'\, \grad l(\ol{X}(Y)_s) \, ds + U_{T,t}(Y)'\, \grad \Psi(\ol{X}(Y)_T), \, h \right\rangle_{\cX^2} 
	= \langle \grad_Y \wt J_T(t,Y), \, h \rangle_{\cX^2},
	\nn
	\end{align}
	in which $\grad_Y \wt J_T(t,Y)$ is as per the lemma statement. Hence, the right-hand equality in \er{eq:D-J-of-Y} holds.
	
	It may be verified that $(t,Y)\mapsto D \wt J_T(t,Y)$ is continuous. In particular, by inspection of \er{eq:D-J-of-Y} and Lemma \ref{lem:X-of-Y-cts} that $(t,Y)\mapsto D_t \wt J_T(t,Y)$ is continuous, i.e. $l\in C^3(\cX^2;\R)$ by \er{eq:ass-M-V}, \er{eq:f}, and $t\mapsto\ol{X}(Y)_t$ and $Y\mapsto\ol{X}(Y)_t$ are both continuous. Similarly, by inspection of \er{eq:D-J-of-Y}, \er{eq:X-grad-J}, $(t,Y)\mapsto D_Y \wt J_T(t,Y)$ is continuous if $Y\mapsto U_{s,t}(Y)$ is continuous, uniformly in $s\in[t,T]$. This uniform continuity property follows by Lemma \ref{lem:U-dble}.
	
	Twice continuous {\Frechet} differentiability follows similarly, via \er{eq:ass-M-V}, \er{eq:D-X-of-Y}, \er{eq:D-J-of-Y}, and Lemma \ref{lem:U-dble}.
\end{proof}

The next result describes Riesz representations of the cost functional, and an auxiliary statement may be found in Appendix B.
 
\begin{proposition}
\label{prop:p-cost-dble}
Given $T\in\R_{>0}$, $t\in[0,T)$, the maps $x\mapsto \bar J_T(t,x,p)$ and $p\mapsto\bar J_T(t,x,p)$ of \er{eq:p-cost} are {\Frechet} differentiable with derivatives given by
\begin{gather}
	\label{eq:D-J-of-p}
	D_x \bar J_T(t,x,p) \in\bo(\cX;\R),
	\qquad D_p \bar J_T(t,x,p) \in \bo(\cX;\R),
	\\
	D_x \bar J_T(t,x,p) \, h = \langle \grad_x \bar J_T(t,x,p), \, h \rangle,
	\qquad 
	D_p \bar J_T(t,x,p) \, h = \langle \grad_p \bar J_T(t,x,p), \, h \rangle,
	\nn
\end{gather}
for all $x,p,h\in\cX$, in which $\grad_x \bar J_T(t,x,p)$ and $\grad_p \bar J_T(t,x,p)$ are the respective Riesz representations
\begin{align}
	\begin{cases}
	\label{eq:p-grad-J}
	\grad_x \bar J_T(t,x,p) 
	\doteq 
	( \ba{cc} \op{I} & 0 \ea ) \, \grad_Y \wt J_T\left(t, Y_p(x) \right),\\
	\grad_p \bar J_T(t,x,p) 
	\doteq 
	( \ba{cc} 0 & \op{I} \ea ) \, \grad_Y \wt J_T\left(t, Y_p(x) \right),
	\end{cases}
\end{align}
and $\grad \wt J_T(t,\cdot)$, $Y_p(x)$ are as per \er{eq:X-grad-J}, \er{eq:f}.

Moreover, given 
\begin{align}
	\label{eq:X-s-components}
	\left( \ba{c} \bar x_s \\ \bar p_s \ea \right)
	& \doteq \ol{X}(Y_{x}(p))_s\in\cX^2, 
	\qquad
	\zeta_s \doteq 
	\left( \ba{c}
		\grad_x \bar J_T(s,\bar x_s,\bar p_s) 
		\\
		\grad_p \bar J_T(s,\bar x_s,\bar p_s)
	\ea \right)\in\cX^2,
	\qquad
	\quad s\in[t,T],
\end{align}
the map $s\mapsto\zeta_s$ satisfies
\begin{align}
	\zeta_{s} 
	& = U_{T,s}(Y_x(p))' \, \grad \Psi(\ol{X}(Y_{x}(p))_T) 
	+ \int_s^T U_{\sigma,s}(Y_x(p))' \, \grad l(\ol{X}(Y_{x}(p))_\sigma) \, d\sigma,
	\label{eq:zeta-dynamics}
\end{align}
for all $s\in[t,T]$. Equivalently, $\zeta_s = \left( \ba{c} \bar p_s - \pi_s \\ \xi_s \ea \right)$ for all $s\in[t,T]$, where $s\mapsto \left( \ba{c} \pi_s \\ \xi_s \ea \right)$ satisfies the final value problem
\begin{align}
	\left( \ba{c}
		\dot\xi_s
		\\
		\dot\pi_s
	\ea \right)	
	& =
	\left( \ba{cc}
		0 & -\op{M}^{-1}
		\\
		\grad^2 V(\bar x_s) & 0
	\ea \right) \left( \ba{c}
		\xi_s
		\\
		\pi_s
	\ea \right)
	= A_s \left( \ba{c}
		\xi_s
		\\
		\pi_s
	\ea \right),
	\qquad s\in[t,T],
	\nn\\
	\left( \ba{c}
		\xi_T
		\\
		\pi_T
	\ea \right) & =
	\left( \ba{c}
		0
		\\
		\bar p_T - \grad\psi(\bar x_T)
	\ea \right).
	\label{eq:FVP}
\end{align}
\end{proposition}

\begin{proof}
	Fix arbitrary $T\in\R_{\ge 0}$, $t\in[0,T)$, $x,p,h\in\cX$. By \er{eq:p-cost}, \er{eq:X-cost},
	$\bar J_T(t,x,p) = \wt J_T\left( t, Y_x(p) \right)$,
	and note that the maps $x\mapsto Y_x(p)$ and $p\mapsto Y_x(p)$ are {\Frechet} differentiable with respective derivatives given by $D_x Y_x(p) \, h = ( \ba{cc} \op{I} & 0 \ea )' \, h$ and $D_p Y_x(p) \, h = ( \ba{cc} 0 & \op{I} \ea )' \, h$, with $0, \op{I}\in\bo(\cX)$ denoting the zero and identity maps. Applying the chain rule, and Proposition \ref{prop:X-cost-dble},
	\begin{align}
	D_x \bar J_T(t,x,p) \, h 
	= D_Y \wt J_T( t, Y_x(p) ) \, D_x Y_x(p) \, h
	& = \langle \grad_Y \wt J_T( t, Y_x(p) ), D_p Y_x(p) \, h \rangle_{\cX^2}
	\nn\\
	& = \langle \left( \ba{cc} \op{I} & 0 \ea \right) \grad_Y \wt J_T( t, Y_x(p) ), \, h \rangle,
	\nn
	\end{align}
	yielding the first asserted derivative in \er{eq:p-grad-J}, with the other asserted derivative following similarly.
	
	For the remaining assertions \er{eq:zeta-dynamics}, \er{eq:FVP}, given \er{eq:X-s-components}, note that
	\begin{align}
	\zeta_s
	& \doteq \left( \ba{c}
	\grad_x \bar J_T(s,\bar x_s,\bar p_s) 
	\\
	\grad_p \bar J_T(s,\bar x_s,\bar p_s)
	\ea \right)
	= \grad_Y \wt J_T(s,Y_{\bar x_s}(\bar p_s))
	= \grad_Y \wt J_T(s,\ol{X}(Y_x(p))_s),
	\nn
	\end{align}
	so that \er{eq:zeta-dynamics} follows by Proposition \ref{prop:X-cost-dble}. As $\op{U}_{T,T} = \op{I}$, and $\Psi$ is as per \er{eq:f}, note by \er{eq:X-cost} that
	\begin{align}
	\zeta_{T} 
	& = \grad\Psi(\ol{X}(Y)_T) 
	= \left( \ba{c} 
	\grad\psi(\bar x_T)
	\\
	0
	\ea \right),
	\nn
	\end{align}
	i.e. the terminal condition in \er{eq:FVP} holds.
	Note further that $s\mapsto\zeta_s$ of \er{eq:zeta-dynamics} is differentiable, with the Leibniz integral rule implying that
	\begin{align}
	\dot\zeta_{s}
	& = (\ts{\pdtone{}{s}} U_{T,s}(Y))' \, \grad \Psi(\ol{X}(Y)_T) 
	- \grad l(\ol{X}(Y)_s) 
	+ \int_s^T (\ts{\pdtone{}{s}} U_{\sigma, s}(Y))' \, \grad l(\ol{X}(Y)_\sigma) \, d\sigma
	\nn\\
	& = -A_s' \, \zeta_s - \grad l(\ol{X}(Y)_s) 
	= \left( \ba{cc}
	0 & -\grad^2 V(\bar x_s)
	\\
	\op{M}^{-1} & 0
	\ea \right) \zeta_s + 
	\left(\ba{c}
	\grad V(\bar x_s) 
	\\
	-\op{M}^{-1}\, \bar p_s
	\ea \right)\!,
	\label{eq:p-stat-dynamics}
	\end{align}
	for all $s\in(t,T)$. Define $s\mapsto\pi_s, \, \xi_s$ via $\left( \ba{c} \bar p_s - \pi_s \\ \xi_s \ea \right)\doteq \zeta_s = \left( \ba{c}
	\grad_x \bar J_T(s,\bar x_s,\bar p_s) 
	\\
	\grad_p \bar J_T(s,\bar x_s,\bar p_s)
	\ea \right)\in\cX^2$, so that
	\begin{align}
	\dot\xi_s
	& = \op{M}^{-1}( \bar p_s - \pi_s) - \op{M}^{-1}\, \bar p_s = -\op{M}^{-1}\, \pi_s,
	\nn\\
	\dot\pi_s
	& = \dot {\bar p}_s - [\dot {\bar p}_s - \dot\pi_s]
	= \grad V(\bar x_s) - [ -\grad^2 V(\bar x_s)\, \xi_s + \grad V(\bar x_s) ]
	= \grad^2 V(\bar x_s) \, \xi_s,
	\nn
	\end{align}
	for all $s\in[t,T]$. That is, the ODE in \er{eq:FVP} also holds.
\end{proof}
\begin{remark}
	An auxiliary statement of Proposition \ref{prop:p-cost-dble}.
	\hfill{$\square$}
\end{remark}


\subsection{Characterization of stationary trajectories}

It may now be demonstrated, via the following lemma and theorem, that existence of the argstat in \er{eq:stat-p-value} corresponds to existence of the argstat in \er{eq:stat-u-value}, under a condition of existence of an argstat along trajectories. An equivalent formulation, involving a TPBVP, is subsequently stated as a corollary.

\begin{lemma}
\label{lem:grad-p-iff-grad-x}
Let $T\in\R_{>0}$, $t\in[0,T)$, $x,p\in\cX$, and $(\bar x_s, \bar p_s) \doteq \ol{X}(Y_p(x))_s$ for all $s\in[t,T]$. Then the following statements are equivalent:
\begin{enumerate}
	\item[(i)] $\bar p_s\in\argstat_{p\in\cX} \bar J_T(s,\bar x_s, p)$ for all $s\in[t,T]$;
	
	\item[(ii)] $\bar p_s = \grad_x \bar J_T(s,\bar x_s,\bar p_s)$ for all $\ s\in[t,T].$
\end{enumerate}
\end{lemma}
\begin{proof}
Fix $T\in\R_{>0}$, $t\in[0,T)$, $x,p\in\cX$, and $Y\doteq Y_x(p)\in\cX^2$. 

Suppose $(i)$ holds true, i.e. $0 = \grad_p \bar J_T(s,\bar x_s, \bar p_s) \doteq \xi_s$ for all $s\in[t,T]$. Then, applying Proposition \ref{prop:p-cost-dble}, i.e. \er{eq:FVP}, $0 = \dot\xi_s = -\op{M}^{-1}\, \pi_s$, so that $0 = \pi_s = \bar p_s - \grad_x \bar J_T(s,\bar x_s,\bar p_s)$ for all $s\in[t,T]$. That is, $(ii)$ holds.

Alternatively, suppose $(ii)$ holds true, i.e. $\bar p_s = \grad_x \bar J_T(s,\bar x_s,\bar p_s)$ for all $s\in[t,T]$. Then, $0 = \bar p_s - \grad_x \bar J_T(s,\bar x_s,\bar p_s) \doteq \pi_s$ for all $s\in[t,T]$, so that $\dot\xi_s = -\op{M}^{-1}\, \pi_s = 0$ for all $s\in[t,T]$ by \er{eq:FVP}. Moreover, the terminal condition $\xi_T = \grad_p \bar J_T(T,\bar x_T,\bar p_T) = \grad_p \psi(\bar x_T) = 0$ holds by definition of \er{eq:p-cost}, so that $0 = \xi_s = \grad_p \bar J_T(s,\bar x_s, \bar p_s)$ for all $s\in[t,T]$, again by \er{eq:FVP}. That is, $(i)$ holds.
\end{proof}

\begin{theorem}
\label{thm:argstat-existence}
Let $T\in\R_{\ge 0}$, $t\in[0,T]$, and $x\in\cX$. Then the following statements are equivalent:
\begin{enumerate}
	\item[(i)] there exists $\ \bar u\in\argstat_{u\in\cU} J_T(t,x,u)$;

	\item[(ii)] there exists $\ \bar p\in\argstat_{p\in\cX} \bar J_T(t,x,p) $ such that $(\bar x_s, \bar p_s) \doteq \ol{X}(Y_{\bar p}(x))_s \text{ of \er{eq:X-s-components} satisfies}$
\begin{align}  
\bar p_s\in\argstat_{p\in\cX} \bar J_T(s,\bar x_s, p) \ \forall \ s\in[t,T],
\nn
\end{align}
\end{enumerate}

in which $J_T$, $\bar J_T$ are as per \er{eq:cost}, \er{eq:stat-u-value}, \er{eq:p-cost}, \er{eq:stat-p-value}.

Moreover, if (i) holds true, then $\bar u$ satisfies
\begin{align}
	\bar u_s 
	& = -\op{M}^{-1} \left( \ba{cc} 0 & \op{I} \ea \right) \ol{X}(Y_{\bar p}(\bar x))_s
	= -\op{M}^{-1}\, \bar p_s
	\quad \forall \ s\in[t,T].
	\label{eq:argstat-existence-u}
\end{align}
\end{theorem}
\begin{proof}
Fix arbitrary $T\in\R_{\ge 0}$, $t\in[0,T]$, and $x\in\cX$. 

{\em $(ii)\Longrightarrow (i)$}
Suppose $\bar p\in\argstat_{p\in\cX} \bar J_T (t,x,p)$ exists such that $(\bar x_s, \bar p_s)\doteq \ol{X}(Y_{\bar p}(x))_s$ via \er{eq:X-s-components} satisfies $\bar p_s \in\argstat_{p\in\cX} \bar J_T(s,\bar x_s,p)$ for all $s\in[t,T]$. Applying Lemma \ref{lem:grad-p-iff-grad-x}, $\bar p_s = \grad_x \bar J_T(s,\bar x_s, \bar p_s)$ for all $s\in[t,T]$, and in particular, $\bar p_T = \grad_x \bar J_T(T, \bar x_T, \bar p_T) = \grad_x \psi(\bar p_T)$. Hence, $s\mapsto \ol{X}(Y_{\bar p}(x))_s$, $s\in[t,T]$, solves the TPBVP \er{eq:char} with $\bar x_t = x$. Theorem \ref{thm:stat-u-char} subsequently implies the existence of $\argstat_{u\in\cU[t,T]} J_T(t,x,u)$, and its explicit form \er{eq:argstat-existence-u}.\\[-2mm]

{\em $(i)\Longrightarrow (ii)$} Suppose there exists 
$
	\bar u\in\argstat_{u\in\cU[t,T]} J_T(t,x,u)
$.
By Theorem \ref{thm:stat-u-char}, and subsequently Lemma \ref{lem:X-classical}, there exists a unique classical solution to TPBVP \er{eq:char} with $\bar x_t = x$ fixed. Denote $\bar p \doteq \bar p_t$ and $(\bar x_s,\bar p_s) \doteq \ol{X}(Y_{\bar p}(x))_s$ for all $s\in[t,T]$, and note by \er{eq:char} that $\bar p_T = \grad\psi(\bar x_T)$. Hence, FVP \er{eq:FVP}  has the trivial solution as its unique solution, i.e. $(\xi_s, \pi_s) = 0$ for all $s\in[t,T]$. Hence, recalling \er{eq:X-s-components}, \er{eq:zeta-dynamics}, \er{eq:FVP},
$$
	\zeta_s = \left( \ba{c} \grad_x \bar J_T(s,\bar x_s,\bar p_s) \\ \grad_p \bar J_T(s,\bar x_s, \bar p_s) \ea \right)
	= \left( \ba{c} \bar p_s - \pi_s \\ \xi_s \ea \right) = \left( \ba{c} \bar p_s \\ 0 \ea \right),
	\qquad 
	s\in[t,T],
$$
so that
$0 = \grad_p \bar J_T(s,\bar x_s,\bar p_s)$, i.e. $\bar p_s \in\argstat_{p\in\cX} \bar J_T(s,\bar x_s,p)$, for all $s\in[t,T]$.
\end{proof}

Theorem \ref{thm:argstat-existence} may also be stated with respect to the TPBVP 
\begin{align}
	\left( \ba{c} 
		\dot \xi_s
		\\
		\dot \pi_s
	\ea \right)
	 = \left( \ba{cc}
		0  -\op{M}^{-1} 
		\\
		\grad^2 V(\bar x_s)  0
	\ea \right) \left( \ba{c} 
		\dot \xi_s
		\\
		\dot \pi_s
	\ea \right),
	\nn\\
	\xi_t = 0 = \xi_T,
	\quad \pi_T = \bar p_T - \grad\psi(\bar x_T),
	\label{eq:TPBVP-stat-p}
\end{align}
in which $s\mapsto (\bar x_s, \bar p_s)$ is as per \er{eq:X-s-components}, for all $s\in[t,T]$. This is formalized below as a corollary.

\begin{corollary}
\label{cor:argstat-existence}
Let $T\in\R_{\ge 0}$, $t\in[0,T]$, and $x\in\cX$. Then the following statements are equivalent:

\begin{enumerate}
	\item[(i)] there exists $ \ \bar u\in\argstat_{u\in\cU} J_T(t,x,u) $;
	
	\item[(ii)] there exists $\bar p\in\argstat_{p\in\cX} \bar J_T(t,x,p) $ such that $\text{TPBVP \er{eq:TPBVP-stat-p}}$ has a unique solution, and it satisfies $ \pi_T = 0,$
\end{enumerate}

in which $J_T$, $\bar J_T$ are as per \er{eq:cost}, \er{eq:stat-u-value}, \er{eq:p-cost}, \er{eq:stat-p-value}.
\end{corollary}

\begin{remark}
As indicated, Corollary \ref{cor:argstat-existence} uses a condition concerning the TPBVP \er{eq:TPBVP-stat-p} that is equivalent to the trajectory based argstat condition appearing in Theorem \ref{thm:argstat-existence}. This condition can be re-expressed via the operator
\begin{align}
	\nn
	& \wh U_{s,t} = \left( \ba{cc}
			\wh U_{s,t}^{11} & \wh U_{s,t}^{12}
			\\
			\wh U_{s,t}^{21} & \wh U_{s,t}^{22}
		\ea \right)\in\bo(\cX^2),
\end{align}
which is defined as an element of the two-parameter family generated by $-A_s'\in\bo(\cX^2)$, $s\in[t,T]$, via  \er{eq:generator}. By definition of $\wh U_{T,t}$, and by inspection of \er{eq:TPBVP-stat-p}, 
\begin{align}
	\left(\ba{c} \xi_s \\ \pi_s \ea \right)
	& = \wh U_{T,t}' \left( \ba{c} \xi_T \\ \pi_T \ea \right),
	\ s\in[t,T],
	\nn
\end{align}
so that the boundary conditions in \er{eq:TPBVP-stat-p} require
\begin{align}
	0 & = (\wh U_{T,t}^{21})' \, \pi_T.
	\nn
\end{align}
Hence, the requirement that $\pi_T = 0$ in the statement of Corollary \ref{cor:argstat-existence} amounts to an invertibility requirement for $(\wh U_{T,t}^{21})'\in\bo(\cX)$, i.e. perturbations in the terminal costate map bijectively to perturbations in the initial state. This type of condition arises in the application of the classical method of characteristics to optimal control problems, and so is unsurprising, see for example \cite{S:95}.
\hfill{$\square$}
\end{remark}

\begin{lemma}
\label{lem:verify-condition-J}
Given any $T\in\R_{>0}$, $t\in[0,T)$, $x,p\in\cX$, and $(\bar x_s,\bar p_s) \doteq \ol{X}(Y_p(x))_s$ for all $s\in[t,T]$,
\begin{align}
	\begin{gathered}
	0 = \grad_p \bar J_T(t,\bar x_t, \bar p_t)
	\\ 
	\bar p_t = \grad_x \bar J_T(t,\bar x_t,\bar p_t)
	\end{gathered} 
	& \qquad\Longleftrightarrow\qquad
	\begin{gathered}
	0 = \grad_p \bar J_T(s,\bar x_s, \bar p_s) \\ \forall \ s\in[t,T]
	\end{gathered}
	\qquad\Longleftrightarrow\qquad
	\begin{gathered}
		\bar p_s = \grad_x \bar J_T(s,\bar x_s,\bar p_s)
		\\
		\forall \ s\in[t,T].
	\end{gathered}
	\label{eq:verify-condition-J}
\end{align}
\end{lemma}
\begin{proof}
Fix arbitrary $T\in\R_{>0}$, $t\in[0,T)$, $x,p\in\cX$. By Proposition \ref{prop:p-cost-dble}, and in particular \er{eq:X-s-components}, \er{eq:zeta-dynamics}, \er{eq:FVP},
\begin{align}
	\left(\ba{c} \xi_s \\ \pi_s \ea \right)
	& \doteq \left( \ba{c} 
			\grad_p \bar J_T(s,\bar x_s,\bar p_s) \\ 
			\bar p_s - \grad_x \bar J_T(s,\bar x_s,\bar p_s)
		\ea \right),
		\quad s\in[t,T],
	\label{eq:xi-pi}
\end{align}
satisfies the ODE in \er{eq:FVP}. Suppose that the left-hand condition in \er{eq:verify-condition-J} holds, i.e. $0 = (\xi_t, \pi_t)$. Hence, by inspection of the ODE in \er{eq:FVP}, and \er{eq:xi-pi}, $0 = (\xi_s,\pi_s)$ for all $s\in[t,T]$, so that the centre and right-hand conditions in \er{eq:verify-condition-J} simultaneously hold.
Conversely, suppose that either the centre or the right-hand condition in \er{eq:verify-condition-J} holds. Then, by Lemma \ref{lem:grad-p-iff-grad-x}, both the centre and right-hand conditions in \er{eq:verify-condition-J} simultaneously hold. Moreover, by inspection of \er{eq:p-cost},  \er{eq:X-s-components}, selecting $s = t$ in the both the centre and right-hand conditions in \er{eq:verify-condition-J} immediately yields the left-hand condition in \er{eq:verify-condition-J}.
\end{proof}

\begin{theorem}
\label{thm:argstat-equivalence}
Let $T\in\R_{>0}$, $t\in[0,T)$, and $\bar x, \bar p\in\cX$. Then the following statements are equivalent:

\begin{enumerate}
	\item[(i)] $\bar x \in\argstat_{x\in\cX} \left\{ \langle x,\bar p\rangle - \bar J_T(t,x,\bar p) \right\}$ and $\bar p \in\argstat_{p\in\cX} \bar J_T(t,\bar x,p)$;
	
	\item[(ii)] $\bar u\in\argstat_{u\in\cU[t,T]} J_T(t,\bar x,u)$ where
	\begin{align}
		\bar u_s
		\doteq  -\op{M}^{-1} \left( \ba{cc} 
		0 & \op{I}
		\ea \right) \ol{X}(Y_{{\bar p}}(\bar x))_s
		= -\op{M}^{-1}\, \bar p_s,\quad \forall s\in[t,T].
		\label{eq:argstat-equivalence}
	\end{align}
	
\end{enumerate}
\end{theorem}
\begin{proof}
Immediate by Theorem \ref{thm:argstat-existence} and Lemma \ref{lem:verify-condition-J}. 
\end{proof}


\subsection{HJB equation and stationary trajectories}
A verification theorem exists for the cost $\bar J_T$ of \er{eq:p-cost}, \er{eq:stat-p-value}, formulated with respect to an extended Hamiltonian $\bar H:\cX^2\times\cX^2\rightarrow\R$ defined by
\begin{align}
	\bar H(x,p,\pi,\zeta)
	& \doteq -\demi\, \langle p, \, \op{M}^{-1}\, p \rangle + V(x) 
			+ \langle \pi,\, \op{M}^{-1}\, p \rangle - \langle \zeta, \, \grad V(x) \rangle
	\label{eq:H-bar}
\end{align}
for all $x,p,\pi,\zeta\in\cX$. 

\begin{theorem}
\label{thm:verify-J-bar}
Given $T\in\R_{>0}$ suppose there exists a $W\in C([0,T]\times\cX^2;\R)\cap C^1((0,T)\times\cX^2;\R)$ satisfying the partial differential equation and terminal condition
\begin{gather}
\label{eq:verify-J-bar}
\left\{
\begin{aligned}
& -\pdtone{W}{t}(t,x,p) + \bar{H}(x,p,\grad_x W(t,x,p), \grad_p W(t,x,p))=0 &&(t,x,p)\in (0,T)\times \cX^2
\\
&W(T,x,p) = \psi(x), &&(x,p)\in \cX^2,
\end{aligned}
\right.
\end{gather}
in which $\bar H$ is as per \er{eq:H-bar}, and $\psi$ is the terminal cost appearing in \er{eq:cost}, \er{eq:p-cost}. Then, $\bar J_T(t,x,p) = W(t,x,p)$ for all $t\in(0,T)$, $x,p\in\cX$, where $\bar J_T$ is as per \er{eq:p-cost}.

Conversely, $\bar J_T$ of \er{eq:p-cost} always satisfies \er{eq:verify-J-bar}, and it is consequently the unique solution.
\end{theorem}
\begin{proof}
Fix $T\in\R_{>0}$, $t\in(0,T)$, and let $W$ be as per the theorem statement. Fix any $x,p\in\cX$. With $\bar X$ as per Lemmas \ref{lem:X-classical}, \ref{lem:X-of-Y-cts}, \ref{lem:X-of-Y-dble}, let $(\bar x_s, \bar p_s) \doteq \ol{X}(Y_p(x))_s$ for all $s\in[t,T]$. Note in particular that $s\mapsto (\bar x_s, \bar p_s)$ is a classical solution of the Cauchy problem \er{eq:Cauchy-dynamics}, \er{eq:X-dynamics}. Hence, by the asserted smoothness of $W$, $s\mapsto W(s,\bar x_s, \bar p_s)$ is differentiable, so that the chain rule and \er{eq:verify-J-bar} yield
\begin{align}
	\ts{\ddtone{}{s}} W(s,\bar x_s,\bar p_s)
	& = \ts{\pdtone{}{s}} W(s,\bar x_s,\bar p_s) + \langle \grad_x W(s,\bar x_s,\bar p_s), \, \dot{\bar x}_s \rangle 
			+ \langle \grad_p W(s,\bar x_s,\bar p_s), \, \dot{\bar p}_s \rangle
	\nn\\
	& = - [ -\ts{\pdtone{}{s}} W(s,\bar x_s,\bar p_s) + \bar{H}(x,p,\grad_x W(s,\bar x_s,\bar p_s), \grad_p W(s,\bar x_s,\bar p_s)) ]
	\nn\\
	& \hspace{20mm}
		+ \bar{H}(x,p,\grad_x W(s,\bar x_s,\bar p_s), \grad_p W(s,\bar x_s,\bar p_s))
	\nn\\
	& \hspace{20mm}
		+ \langle \grad_x W(s,\bar x_s,\bar p_s), \, -\op{M}^{-1}\, \bar p_s \rangle
		+ \langle \grad_p W(s,\bar x_s,\bar p_s), \, \grad V(\bar x_s) \rangle
	\nn\\
	& =  -\demi\, \langle \bar p_s, \, \op{M}^{-1}\, \bar p_s \rangle + V(\bar x_s) 
		+ \langle \grad_x W(s,\bar x_s,\bar p_s),\, \op{M}^{-1}\, \bar p_s \rangle 
		- \langle \grad_p W(s,\bar x_s,\bar p_s), \, \grad V(\bar x_s) \rangle 
	\nn\\
	& \hspace{20mm}
		+ \langle \grad_x W(s,\bar x_s,\bar p_s), \, -\op{M}^{-1}\, \bar p_s \rangle
		+ \langle \grad_p W(s,\bar x_s,\bar p_s), \, \grad V(\bar x_s) \rangle
	\nn\\
	& = -\demi\, \langle \bar p_s, \, \op{M}^{-1}\, \bar p_s \rangle + V(\bar x_s),
	\nn
\end{align}
for all $s\in(t,T)$. Integrating with respect to $s\in(t,T)$, and recalling the boundary condition in \er{eq:verify-J-bar}, subsequently yields
\begin{align}	
	\psi(\bar x_T) - W(t,x,p) = 
	W(T,\bar x_T, \bar p_T) - W(t,x,p)
	& = \int_t^T -\demi\, \langle \bar p_s, \, \op{M}^{-1}\, \bar p_s \rangle + V(\bar x_s) \, ds.
	\nn
\end{align}
Rearranging, and recalling \er{eq:p-cost}, yields the asserted equality $\bar J_T(t,x,p) = W(t,x,p)$. Recalling that $t\in(0,T)$, $x,p\in\cX$ are arbitrary yields the first assertion.

For the converse, note by Proposition \ref{prop:X-cost-dble} that $\bar J_T\in C([0,T]\times\cX^2;\R)\cap C^1((0,T)\times\cX^2;\R)$. Fix $x,p\in\cX$. Note by \er{eq:p-cost} that $\bar J_T(T,x,p) = \psi(x)$, so that the terminal condition in \er{eq:verify-J-bar} trivially holds. Fix $t\in(0,T)$, and let $(\bar x_s, \bar p_s) \doteq \ol{X}(Y_p(x))_s$ for all $s\in[t,T]$. Fix $r\in(t,T]$. By \er{eq:p-cost},
\begin{align}
	\bar J_T(t,x,p)
	& = \int_t^r -\demi\, \langle \bar p_s, \, \op{M}^{-1}\, \bar p_s \rangle + V(\bar x_s) \, ds
			+ \int_r^T -\demi\, \langle \bar p_s, \, \op{M}^{-1}\, \bar p_s \rangle + V(\bar x_s) \, ds + \psi(\bar x_T)
	\nn\\
	& = \int_t^r -\demi\, \langle \bar p_s, \, \op{M}^{-1}\, \bar p_s \rangle + V(\bar x_s) \, ds + \bar J_T(r,\bar x_r, \bar p_r).
	\nn
\end{align}
Dividing through by $r-t$ and sending $r\rightarrow t^+$, yields
\begin{align}
	-\demi\, \langle p, \, \op{M}^{-1}\, p \rangle + V(x)
	& = 
	\ts{\ddtone{}{t}} \bar J_T(t,\bar x_t, \bar p_t)
	\nn\\
	& = \ts{\pdtone{}{t}} \bar J_T(t,x,p) + \langle \grad_x \bar J_T(t,x,p), \, -\op{M}^{-1}\, p \rangle 
	+ \langle \grad_p \bar J_T(t,x,p), \, \grad V(x) \rangle,
	\nn
\end{align}
i.e. $\bar J_T$ satisfies \er{eq:verify-J-bar}, and uniqueness follows by the first assertion.
\end{proof}

\begin{theorem}
\label{thm:verify}
Given $T\in\R_{>0}$, suppose there exists a $W\in C([0,T]\times\cX^2;\R)\cap C^1((0,T)\times\cX^2;\R)$ satisfying the partial differential equation and terminal condition \er{eq:verify-J-bar}. Suppose further that, given $t\in[0,T]$ and $\bar x\in\cX$, there exists a $\bar p\in\cX$ such that 
\begin{align}
	\bar p \in \argstat_{p\in\cX} W(t,\bar x, p), 
	\qquad
	\bar x\in \argstat_{x\in\cX} \left\{ \langle x, \bar p \rangle - W(t,x,\bar p) \right\}.
	\label{eq:verify-conditions} 
\end{align}
Then, there exists a $\bar u\in \argstat_{u\in\cU[t,T]} J_T(t,\bar x,u)$ such that
\begin{gather}
	\bar x_s = \bar x + \int_t^s \bar u_\sigma \, d\sigma,
	\quad
	\bar p_s \in\argstat_{p\in\cX} W(s,\bar x_s,p),
	\nn\\
	\bar u_s 
		= -\op{M}^{-1} \, \grad_x W(s,\bar x_s, \bar p_s) = -\op{M}^{-1}\, \bar p_s,
	\label{eq:verify-u-stat}
\end{gather}
for all $s\in[t,T]$.
\end{theorem}
\begin{proof}
Fix $T\in\R_{>0}$, $W\in C([0,T]\times\cX^2;\R)\cap C^1((0,T)\times\cX^2;\R)$, $t\in[0,T]$, $x\in\cX$, as per the theorem statement. Suppose that $\bar p\in\cX$ exists such that \er{eq:verify-conditions} holds. Observe by Theorem \ref{thm:verify-J-bar} that $\bar J_T \equiv W$. 
Hence, by \er{eq:verify-conditions} and Theorem \ref{thm:argstat-equivalence}, there exists $\bar u\in\argstat_{u\in\cU[t,T]} J_T(t,\bar x,u)$ satisfying $\bar u_s = -\op{M}^{-1}\, \bar p_s$ for all $s\in[t,T]$. Moreover, by Lemma \ref{lem:verify-condition-J}, $\bar p_s = \grad_x \bar J_T(s,\bar x_s,\bar p_s) = \grad_x W(s,\bar x_s,\bar p_s)$, and $0 = \grad_p \bar J_T(s,\bar x_s,\bar p_s)$ for all $s\in[t,T]$, so that \er{eq:verify-u-stat} holds.
\if{false}

??

 implies the left-hand property in \er{eq:verify-condition-J}, so that by Lemma \ref{lem:verify-condition-J}, the right-hand property in \er{eq:verify-condition-J} holds. In particular, $\bar p_s \in\argstat_{p\in\cX} \bar J_p (s,\bar x_s,p)$ for all $s\in[t,T]$. Hence, by Theorem \ref{thm:argstat-existence}, $\bar u\in\argstat_{u\in\cU} J_T(t,x,u)$ exists, and its representation \er{eq:verify-u-stat} follows by \er{eq:Cauchy-dynamics}.

\fi
\end{proof}


\begin{remark}
\
\begin{enumerate}
	\item[(i)] The characteristic system associated with \er{eq:H-bar} is
	\begin{align}
	& \left\{ \begin{aligned}
	\dot x_s
	& = -\grad_\pi \bar H(x_s,p_s,\pi_s,\zeta_s) = -\op{M}^{-1}\, p_s,
	& x_t = x,
	\\
	\dot p_s
	& = -\grad_\zeta \bar H(x_s,p_s,\pi_s,\zeta_s) = \grad V(x_s),
	& p_t = p,
	\\
	\dot \pi_s
	& = \grad_x \bar H(x_s,p_s,\pi_s,\zeta_s) = \grad V(x_s) - \grad^2 V(x_s)\, \zeta_s,
	\\
	\dot \zeta_s
	& = \grad_p \bar H(x_s,p_s,\pi_s,\zeta_s) = -\op{M}^{-1}\, (p_s - \pi_s),
	\end{aligned} \right.
	\label{eq:char-H-bar}
	\end{align}
	for all $s\in[0,T]$, in which $\pi_s = \grad_x W(s,x_s,p_s)$ and $\zeta_s = \grad_p W(s,x_s,p_s)$. By adopting the condition \er{eq:verify-conditions}, it may be noted that $p_s = \pi_s$, $\zeta_s = 0$, and $\bar H(x_s,p_s,\pi_s,\zeta_s) = H(x_s,p_s)$, for all $s\in[t,T]$.

	\item[(ii)] It may be noted that under the stated assumptions, if $\argstat_{p\in \cX}\bar J_T(t,x,p)$ is convex for a.e. $t\in(0,T)$ and all $x\in \cX$, then from the $C^1$ regularity of $\bar J_T$ (c.f. Proposition \ref{prop:X-cost-dble}) it follows that $\stat_{p\in \cX}\bar J_T(t,x,p)$ is single-valued for a.e. $t\in(0,T)$ and all $x\in \cX$ and $\ol W_T(\cdot,\cdot)$ is the (viscosity) solution of the HJB equation
	\begin{gather}
	\left\{
	\begin{aligned}
	& -\pdtone{W}{t}(t,x) + {H}(x,\grad_x W(t,x))=0 &&(t,x)\in [0,T]\times \cX
	\\
	&W(T,x) = \psi(x), &&x\in \cX,
	\end{aligned}
	\right.
	\label{viscosity}
	\end{gather}
	where $H$ is the Hamiltonian defined in (\ref{eq:H}). Indeed, since it is known that the minimal selection $W_T$ is the unique viscosity solution of (\ref{viscosity}) (c.f. \cite{CS:04,S:95}), it is sufficient to show that
	\begin{align}
	W_T(t,x)=\ol W_T(t,x)\quad \forall t\in[0,T], \,\forall x\in \cX.
	\label{claim}
	\end{align}
	Fix $t\in [0,T]$ and all $x\in \cX$. Then, letting $L(y,u)=\demi \langle  { u}, \, {M}\,  { u} \rangle - V(y)$ and using that \(H(y, p)=p \nabla_pH(y, p)-L\left(y, \nabla_pH(y, p)\right)\) for any $p,y\in \cX$, applying Theorem \ref{thm:stat-u-char}, and keeping the same notation, we have that there exists $p\in \cX$ satisfying
	\begin{align}
	J_T(t, x,\bar u)&=\psi(\bar x_T)+\int_t^T L(\bar x_s,{\bar u}_s) ds
	=\psi(\bar x_T)+\int_{t}^{T} L\left(\bar x_s, \nabla_pH(\bar x_s, \bar p_s)\right) d s
	\nn\\
	&=\psi(\bar x_T)+\int_{t}^{T}\left(-H(\bar x_s, \bar p_s)+\langle \bar p_s , \nabla_p H(\bar x_s, \bar p_s)\rangle \right) d s=\bar J_T(t, x, p),
	\nn
	\end{align}
	where $(\bar x,\bar p)$ is the Hamiltonian flow  (\ref{eq:Cauchy-dynamics}) with initial condition $(x,p)$.	Hence, the minimal selection coincide with $\ol W_T$, and (\ref{claim}) follows.
	\hfill{$\square$}
\end{enumerate}
\end{remark}


\section{A one-dimensional example}
\label{sec:example}
A one-dimensional linear mass-spring system consists of a mass $\op{M}\doteq m\in\R_{>0}$ located at position $x\in\R$ whose motion is a consequence of a quadratic potential field $V:\R\rightarrow\R_{\ge 0}$, $V(x)\doteq \demi\, \kappa\, x^2$, $x\in\R$. For simplicity, supposes that the terminal velocity $v\in\R$ of this mass is of interest. The terminal cost $\psi:\R\rightarrow\R$ in \er{eq:p-cost}, \er{eq:X-cost} is consequently defined by $\psi(x) \doteq -m\, v\, x$ for all $x\in\R$. Observe by inspection that \er{eq:ass-M-V} holds, with $K \doteq 2\, \kappa$. With a view to demonstrating the existence of an explicit solution to \er{thm:verify-J-bar}, fix $t\in[0,T]$, and define
\begin{gather}
	\Sigma
	\doteq \left( \ba{cc}
		\kappa & 0 \\
		0 & -\ts{\frac{1}{m}}
	\ea \right),
	\qquad
	\Gamma
	\doteq \left( \ba{cc}
			0 & -\ts{\frac{1}{m}}
			\\
			\kappa & 0 
		\ea \right).
	\label{eq:Gamma-Sigma}
\end{gather}
Recalling \er{eq:H-bar}, \er{eq:Gamma-Sigma}, note that the PDE \er{eq:verify-J-bar} may be compactly written as 
\begin{align}
	0
	& = - \pdtone{W}{s} (s,x,p) + \demi \left\langle \left( \ba{c} x \\ p \ea \right), \, \Sigma \left( \ba{c} x \\ p \ea \right) \right\rangle
		- \left\langle \left( \ba{c} \grad_x W(s,x,p) \\ \grad_p W(s,x,p) \ea \right), \, 
								\Gamma \left( \ba{c} x \\ p \ea \right) \right\rangle
	\nn\\
	& = - \pdtone{W}{s} (s,Y) + \demi\, \langle Y, \, \Sigma\, Y \rangle - \langle \grad_Y W(s,Y), \, \Gamma\, Y \rangle
	\label{eq:compact-PDE}
\end{align}
for all $s\in(t,T)$, $Y\doteq Y_x(p)\in\cX^2$. Define $\wt W:[t,T]\times\cX^2\rightarrow\R$ by
\begin{gather}
	\wt W(s,Y)
	\doteq \demi\, \langle Y, \, P_s\, Y \rangle + \langle Q_s, \, Y\rangle,
	\label{eq:ex-W}
	\\
	P_s
	\doteq - \int_s^T \exp(\Gamma' \, (\sigma-s))\, \Sigma\, \exp(\Gamma \, (\sigma-s))\, d\sigma,
	\quad
	Q_s 
	\doteq \exp(\Gamma' \, (T-s))\, \left( \ba{c} -m\, v \\ 0 \ea \right),
	\label{eq:ex-P-Q}
\end{gather}
for all $s\in(t,T)$, $Y\in\cX^2$. Applying Leibniz, note that 
\begin{align}
	\dot P_s = \Sigma - \Gamma' \, P_s - P_s\, \Gamma, \qquad
	\dot Q_s = -\Gamma' \, Q_s,
	\label{eq:ex-P-Q-dot}
\end{align}
for all $s\in(t,T)$. Hence, differentiating \er{eq:ex-W} yields
\begin{align}
	\ts{\pdtone{}{s}} \wt W(s,Y)
	& = \demi\, \langle Y, \, \dot P_s\, Y \rangle + \langle \dot Q_s, \, Y\rangle, 
	\qquad
	\grad_Y \wt W(s,Y) = P_s\, Y + Q_s.
	\nn
\end{align}
Substituting these derivatives in the right-hand side of \er{eq:compact-PDE}, and applying \er{eq:ex-P-Q-dot}, subsequently yields
\begin{align}
	& - \pdtone{\wt W}{s} (s,Y) + \demi\, \langle Y, \, \Sigma\, Y \rangle - \langle \grad_Y \wt W(s,Y), \, \Gamma\, Y \rangle
	\nn\\
	& \hspace{20mm}
	= -\demi\, \langle Y, \, \dot P_s\, Y \rangle - \langle \dot Q_s, \, Y\rangle + \demi\, \langle Y, \, \Sigma\, Y \rangle
	- \langle P_s\, Y + Q_s, \, \Gamma\, Y \rangle
	\nn\\
	& \hspace{20mm}
	= \demi\, \langle Y, \, [ -\dot P_s + \Sigma - P_s\, \Gamma - \Gamma'\, P_s ]\, Y \rangle
	+ \langle -\dot Q_s - \Gamma' \, Q_s, \, Y \rangle
	\nn\\
	& \hspace{20mm}
	= 0,
	\nn
\end{align}
for all $s\in(t,T)$, $Y\in\cX^2$. Note further that $\wt W(T,x,p) = \wt W(T,Y) = \langle Q_T, \, Y \rangle = -m\, v\, x = \psi(x)$.
Hence, $\wt W$ of \er{eq:ex-W} is a solution of the PDE and terminal condition of \er{eq:verify-J-bar}. Hence, by Theorem \ref{thm:verify-J-bar}, the cost $\bar J_T(s,x,p)$ of \er{eq:p-cost}, \er{eq:X-cost} is given explicitly by $\bar J_T(s,x,p) = \wt W(s,Y_x(p))$ for all $s\in[t,T]$, $x,p\in\cX$. By diagonalizing $\Gamma$, direct integration of \er{eq:ex-P-Q} yields
\begin{align}
	P_s
	& = \demi \left( \ba{cc}
		-\frac{\kappa}{\omega}\, \sin(2\, \omega\, (T-s)) & 1 - \cos(2\, \omega(T-s))
		\\
		1 - \cos(2\, \omega\, (T-s))  & \frac{1}{m\, \omega}\, \sin(2\, \omega\, (T-s))
	\ea \right),
	\quad \omega\doteq \sqrt{\frac{\kappa}{m}},
	\nn\\
	Q_s
	& = -m\, v\, \left( \ba{c}
		\cos(\omega\, (T-s))
		\\
		-\frac{\omega}{\kappa}\, \sin(\omega\, (T-s))
	\ea \right).
	\nn
\end{align}

\begin{figure}[h]
	\begin{center}
		\vspace{5mm}
		\begin{subfigure}[b]{0.45\textwidth}
			\includegraphics[width=\textwidth,height=55mm]{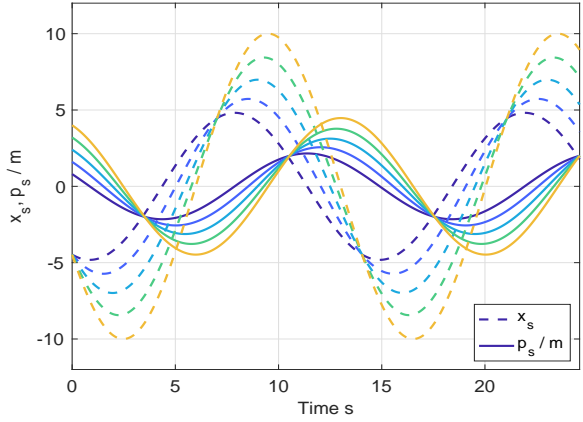}
			\caption{Time evolution.}
			\label{fig:xp-vs-t-2D}
		\end{subfigure}
		\quad\quad
		\begin{subfigure}[b]{0.45\textwidth}
			\includegraphics[width=\textwidth,height=55.5mm]{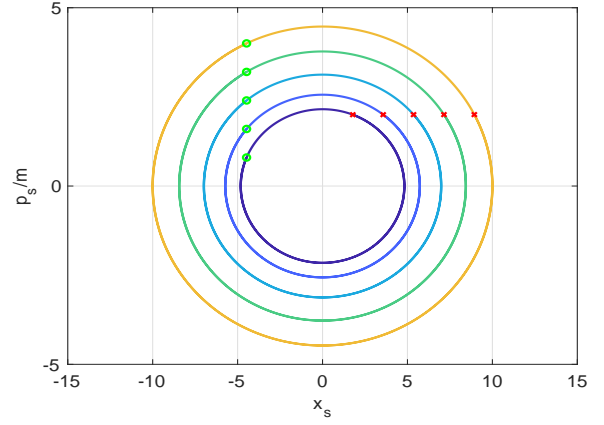}
			\caption{Phase portrait.}
			\label{fig:x-vs-p-2D}
		\end{subfigure}
		\begin{subfigure}[b]{0.6\textwidth}
			\vspace{5mm}
			\includegraphics[width=\textwidth,height=55mm]{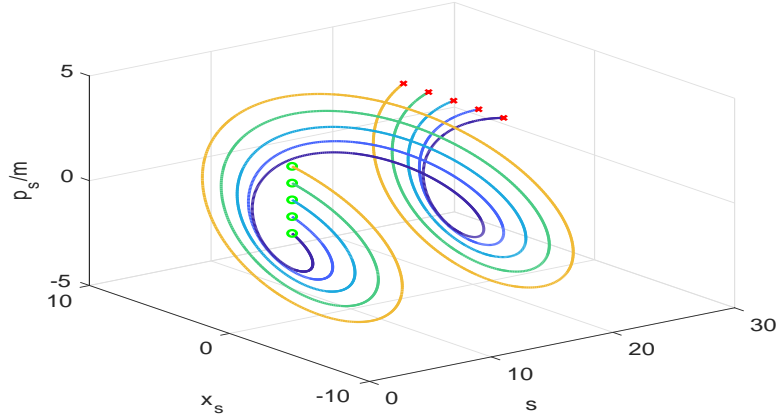}
			\caption{Time evolution of phase portrait.}
			\label{fig:xp-vs-t-3D}
		\end{subfigure}
		
		\caption{State and costate trajectories.}
		
	\end{center}
\end{figure}
With a view to illustrating Theorem \ref{thm:verify}, fix $x\in\R$, and note that
\begin{align}
	\grad_x \wt W(t,x,p)
	& = ( \ba{cc} 1 & 0 \ea ) \, \grad_Y \wt W(t,Y_x(p))
	= ( \ba{cc} 1 & 0 \ea ) \, \left( P_t\, \left( \ba{c} x \\ p \ea \right) + Q_t \right)
	\nn\\
	& = -\ts{\frac{\kappa}{2\, \omega}}\, \sin(2\, \omega\, (T-t))\, x + \demi\, [1 - \cos(2\, \omega(T-t))]\, p
		- m\, v\, \cos(\omega\, (T-t)),
	\label{eq:ex-grad-x-W}
	\\
	\grad_p \wt W(t,x,p) 
	& = ( \ba{cc} 0 & 1 \ea ) \, \grad_Y \wt W(t,Y_x(p))
	= ( \ba{cc} 0 & 1 \ea ) \, \left( P_t\, \left( \ba{c} x \\ p \ea \right) + Q_t \right)
	\nn\\
	& = 
	\demi\, [1 - \cos(2\, \omega\, (T-t))]\, x + \ts{\frac{1}{2\, m\, \omega}}\, \sin(2\, \omega\, (T-t))\, p 
		+ m\, v\, (\ts{\frac{\omega}{\kappa}})\, \sin(\omega\, (T-t)).
	\label{eq:ex-grad-p-W}
\end{align}
Note that $m\, \omega = \sqrt{\kappa\, m} = \frac{\kappa}{\omega}$. Motivated by \er{eq:verify-conditions}, let $\bar p\in\R$ be such that $0 = \grad_p \wt W(t,x,\bar p)$ and $\bar p = \grad_x \wt W(t,x,\bar p)$ via \er{eq:ex-grad-x-W} and \er{eq:ex-grad-p-W}. Applying double angle formulae subsequently yields the system of linear equations
\begin{align}
	& \Omega_t \left( \ba{c} x \sqrt{\kappa\, m} \\ \bar p \ea \right)
	= 
	\Theta_t,	
	\label{eq:Omega-Theta}
\end{align}
in which
\begin{align}
	\Omega_t 
	& \doteq \left( \ba{cc}
		\sin^2(\omega\, (T-t)) & \sin(\omega\, (T-t))\, \cos(\omega\, (T-t))
		\\
		\sin(\omega\, (T-t))\, \cos(\omega\, (T-t)) & \cos^2(\omega\, (T-t))
	\ea \right),
	\
	\Theta_t
	\doteq
	- m\, v \left( \ba{c}
		\sin(\omega\, (T-t))
		\\
		\cos(\omega\, (T-t))
	\ea \right),
	\nn
\end{align}

By inspection, the matrix $(\ba{c|c} \Omega_t & \Theta_t \ea)\in\R^{2\times 3}$ is rank one, i.e. the two equations in \er{eq:Omega-Theta} are linearly dependent. Moreover, some minor manipulations yield
\begin{align}
	& \!\!
	\left\{ \begin{aligned}
		\bar p & = -\sqrt{\kappa\, m}\, \tan(\omega\, (T-t))\, x - m\, v\, \sec(\omega\, (T-t)),
		\!\!
		&&
		\omega\, (T-t) \not\in \left\{ n\, \pi, \, (n+\demi)\, \pi : n\in\Z \right\},
		\\
		\bar p & = (-1)^{n+1}\, m\, v,
		&& \omega\, (T-t) \in \{ n\, \pi : n\in\Z \}, \, x\in\R,
		\\
		\bar p & \text{ arbitrary},
		&& \omega\, (T-t) \in \{ (n + \demi)\, \pi : n\in\Z \}, \, x = (-1)^{n+1}\, (\ts{\frac{v}{\omega}}),
		\\
		\bar p & \text{ does not exist},
		&& \omega\, (T-t) \in \{ (n + \demi)\, \pi : n\in\Z \}, \, x \ne (-1)^{n+1}\, (\ts{\frac{v}{\omega}}).
	\end{aligned} \right.
	\nn
\end{align}
Note in the second case that $\bar p$ must correspond to the desired terminal momentum, with sign determined by whether $T-t$ is a period or half-period of the mass-spring oscillation. In the third and fourth cases, $T-t$ corresponds to a quarter or three quarter period of the mass-spring oscillation, and $\bar p$ is either arbitrary, or does not exist, depending on the specific choice of $x$. An example of the third case, where $\bar p$ is arbitrary, is illustrated in Figures \ref{fig:xp-vs-t-2D}, \ref{fig:x-vs-p-2D} and \ref{fig:xp-vs-t-3D}, for $v \doteq -2$ and $x=(-1)^{4}\, (\frac{v}{\omega}) \approx -4.47$.




\section*{Appendix}

\appendix

\section{Proofs of lemmas \ref{lem:X-classical}, \ref{lem:X-of-Y-cts}, \ref{lem:X-of-Y-dble}, and \ref{lem:U-dble}}\label{proofs_lemmata}

\begin{proof}[Proof of Lemma \ref{lem:X-classical}]
	The proof employs a standard fixed point argument, exploiting global Lipschitz continuity of $f$ of \er{eq:f}, see for example \cite[Theorem 5.1, p.127]{P:83}. We notice that the global Lipschitz continuity of $DV(x)$ directly follows from the second inequality in (\ref{eq:ass-M-V}) for j = 0. 
	\if{false}
	
	Given any $\wt X, X\in\cX^2$, 
	\begin{align}
	& \| f(\wt X) - f(X) \|^2 
	\nn\\
	& = \|\op{M}^{-1}\, (\tilde p - p) \|^2 + \| \grad V(\tilde x) - \grad V(x) \|^2
	\nn\\
	& \le \| \op{M}^{-1}\|_{\bo(\cX)}^2 \, \| \tilde p - p \|^2
	\nn\\
	& \qquad + \sup_{y\in\cX} \| \grad^2 V (y) \|_{\bo(\cX)}^2 \| \tilde x - x \|^2
	\nn\\
	& \le (\ts{\frac{1}{m}})^2 \, \| \tilde p - p \|^2 + (\ts{\frac{K}{2}})^2 \, \| \tilde x - x \|^2
	\le \alpha^2 \, \|\wt X - X\|^2,
	\nn
	\end{align}
	in which $\alpha \doteq \max(\ts{\frac{1}{m}}, \, \ts{\frac{K}{2}})$. That is, $f$ is globally Lipschitz continuous, with Lipschitz constant $\alpha$. 
	
	Fix any $T\in\R_{\ge 0}$, $t\in[0,T]$, $Y\in\cX^2$. Define an operator $\Gamma_Y:C([t,T];\cX^2)\rightarrow C([t,T];\cX^2)$ by
	\begin{align}
	[\Gamma_Y(X)]_s 
	& = Y + \int_t^s f(X_\sigma) \, d\sigma,
	\label{eq:Gamma-X}
	\end{align}
	for all $X\in C([t,T];\cX^2)$, $s\in[t,T]$. Observe by the aforementioned Lipschitz property that
	\begin{align}
	\| [\Gamma_Y(\wt X)]_s - [\Gamma_Y(X)]_s \|
	& \le \int_t^s \| f(\wt X_\sigma) - f(X_\sigma) \| \, d\sigma
	\nn\\
	& \le \alpha\, (s-t)\, \| \wt X - X \|_\infty.
	\nn
	\end{align}
	By induction, it follows that
	\begin{align}
	& \| [\Gamma_Y^k(\tilde X)]_s - [\Gamma_Y^k(X)]_s \|
	\le \frac{\alpha^k\, (s-t)^k}{k!} \, \| \wt X - X\|_\infty
	\nn 
	\end{align}
	for all $k\in\N$, $s\in[t,T]$, so that
	\begin{align}
	& \| \Gamma_Y^k(\tilde X) - \Gamma_Y^k(X) \|_\infty \le \frac{\alpha^k\, (T-t)^k}{k!} \, \| \wt X - X\|_\infty,
	\nn
	\end{align}
	for all $k\in\N$. Note by inspection that there exists a $\bar k\in\N$ sufficiently large such that $\frac{\alpha^k\, (T-t)^k}{k!}<1$ for all $k>\bar k$. Hence, by Banach's contraction mapping principle, see for example \cite{K:78}, there exists a unique fixed point $\ol{X}\in\cX^2$ of $\Gamma_Y$. In particular, by \er{eq:Gamma-X},
	\begin{align}
	\ol{X}_s
	& = Y + \int_t^s f(\ol{X}_\sigma)\, d\sigma
	\nn
	\end{align}
	for all $s\in[t,T]$. Note further that the right-hand side is differentiable, as $f$ is (Lipschitz) continuous and $\ol{X}\in C([t,T];\cX^2)$. Hence, the derivative $s\mapsto\dot{\ol{X}}_s$ exists and is continuous, with $\dot{\ol{X}}_s = f(\ol{X}_s)$, for all $s\in(t,T)$. i.e. $\ol{X}\in C([t,T];\cX^2)\cap C^1((t,T);\cX^2)$ is the unique solution of \er{eq:X-dynamics}, and it is a classical solution.
	
	\fi
\end{proof}

\begin{proof}[Proof of Lemma \ref{lem:X-of-Y-cts}]
	Fix $T\in\R_{\ge 0}$, $t\in[0,T]$, and $Y,h\in\cX^2$. Applying Lemma \ref{lem:X-classical}, there exist unique classical solutions $\ol{X}(Y)$ and $\ol{X}(Y+h)$ to \er{eq:X-dynamics} satisfying respectively $\ol{X}(Y)_t = Y$ and $\ol{X}(Y+h)_t = Y+h$. In integral form,
	\begin{align}
	\label{eq:X-of-Y-s}
	\ol{X}(Y)_s 
	& = Y + \int_t^s f(\ol{X}(Y)_\sigma) \, d\sigma,
	\\
	\ol{X}(Y+h)_s
	& = Y+h + \int_t^s f(\ol{X}(Y+h)_\sigma) \, d\sigma,
	\nn
	\end{align}
	so that
	\begin{align}
	\ol{X}(Y+h)_s - \ol{X}(Y)_s
	& = h + \int_t^s f(\ol{X}(Y+h)_\sigma) - f(\ol{X}(Y)_\sigma) \, d\sigma
	\nn
	\end{align}
	for all $s\in[t,T]$. Consequently, as $f$ is globally Lipschitz by inspection of \er{eq:f},
	\begin{align}	
	\| \ol{X}(Y+h)_s - \ol{X}(Y)_s \| 
	& \le \|h\| + \int_t^s \| f(\ol{X}(Y+h)_\sigma) - f(\ol{X}(Y)_\sigma) \| \, d\sigma
	\nn\\
	& \le \|h\| + \alpha \int_t^s \| \ol{X}(Y+h)_\sigma - \ol{X}(Y)_\sigma \| \, d\sigma
	\nn 
	\end{align}
	in which $\alpha\in\R_{\ge 0}$ is the associated Lipschitz constant. Applying Gronwall's inequality, and recalling the definition of $\|\cdot \|_\infty$, yields
	\begin{align}
	& \| \ol{X}(Y + h) - \ol{X}(Y) \|_\infty
	\le \| h \| \, \exp ( \alpha\, ( T - t) ),
	\nn
	\end{align}
	so that \er{eq:X-of-Y-cts} holds. As $Y,h\in\cX^2$ are arbitrary, continuity is immediate.
\end{proof}

\begin{proof}[Proof of Lemma \ref{lem:X-of-Y-dble}]
	Fix $T\in\R_{\ge 0}$, $t\in[0,T]$, and $\ol{X}\in C(\cX^2;C([t,T];\cX^2)$ as per Lemma \ref{lem:X-of-Y-cts}. Fix $Y\in\cX^2$ and $s\mapsto A(Y)_s$ as per \er{eq:generator}, and note that \er{eq:two-parameter} follows by \cite[Theorem 5.2, p.128]{P:83}. Fix any $h\in\cX^2$, $s\in[t,T]$, and note by inspection of \er{eq:f} that $A(Y)_s = D f(\ol{X}(Y)_s)$. Hence, recalling \er{eq:X-of-Y-s},
	\begin{align}
	& \ol{X}(Y+h)_s - \ol{X}(Y)_s - U_{s,t}(Y)\, h
	= \int_t^s  f(\ol{X}(Y+h)_\sigma) - f(\ol{X}(Y)_\sigma) - A(Y)_\sigma\, U_{\sigma, t}(Y) \, h \, d\sigma
	\nn\\
	& = \int_t^s f(\ol{X}(Y)_\sigma + [\ol{X}(Y+h)_\sigma - \ol{X}(Y)_\sigma]) - f(\ol{X}(Y)_\sigma)
	- D f(\ol{X}(Y)_\sigma)\, U_{\sigma, t}(Y) \, h \, d\sigma
	\nn\\
	& = \int_t^s f(\ol{X}(Y)_\sigma + [\ol{X}(Y+h)_\sigma - \ol{X}(Y)_\sigma]) - f(\ol{X}(Y)_\sigma)
	- D f(\ol{X}(Y)_\sigma)\, [ \ol{X}(Y+h)_\sigma - \ol{X}(Y)_\sigma ]
	\nn\\
	&
	\hspace{10mm}
	+ D f(\ol{X}(Y)_\sigma)\, [ \ol{X}(Y+h)_\sigma - \ol{X}(Y)_\sigma - U_{\sigma, t}(Y) \, h] \, d\sigma.
	\label{eq:pre-D-X-of-Y}
	\end{align}
	Define $\bar I_f:C([t,T];\cX^2)\rightarrow C([t,T];\cX^2)$ by
	\begin{align}
	\label{eq:I-f}
	& \bar I_f(X)_s
	\doteq \int_t^s f(X_\sigma) \, d\sigma
	\end{align}
	for all $X\in C([t,T];\cX^2)$. Note that $Y\mapsto f(Y)$ is twice {\Frechet} differentiable by \er{eq:ass-M-V}, with $D^2 f(Y)\in\bo(\cX^2;\bo(\cX^2)) = \bo(\cX^2\times\cX^2;\cX^2)$ for all $Y\in\cX^2$. Again by \er{eq:ass-M-V}, there exists an $M\in\R_{>0}$ such that
	\begin{align}
	& \sup_{Y\in\cX^2} \| D^2 f(Y) \|_{\bo(\cX^2\times\cX^2;\cX^2)} \le M < \infty.
	\nn
	\end{align}
	Hence, by Taylor's theorem, given $X,\delta\in C([t,T];\cX^2)$,
	\begin{align}
	& \left\| \bar I_f(X + \delta)_s - \bar I_f(X)_s - \int_t^s D f(X_\sigma)\, \delta_\sigma \, d\sigma \right\|
	\le  \int_t^s \| f(X_\sigma + \delta_\sigma) - f(X_\sigma) - D f(X_\sigma)\, \delta_\sigma \| \, d\sigma
	\nn\\
	& = \int_t^s \left\| \left( \int_0^1 (1-\eta) \, D^2 f(X_\sigma + \eta\, \delta_\sigma)\, d\eta \right) (\delta_\sigma, \delta_\sigma)
	\right\| d\sigma
	\nn\\
	& \le \int_t^s \int_0^1 (1-\eta) \, \|D^2 f(X_\sigma + \eta\, \delta_\sigma)\|_{\bo(\cX^2\times\cX^2;\cX^2)} \,
	\| \delta_\sigma\|^2 \, d\sigma
	\le \ts{\frac{M}{2}} \int_t^s \|\delta_\sigma\|^2 \, d\sigma
	\le \ts{\frac{M}{2}}\, (s-t)\, \|\delta\|_\infty^2.
	\nn 
	\end{align}
	That is,
	\begin{align}
	& \left\| \bar I_f(X + \delta) - \bar I_f(X) - \int_t^{(\cdot)} D f(X_\sigma)\, \delta_\sigma \, d\sigma \right\|_\infty
	\le \ts{\frac{M}{2}}\, (T-t)\, \|\delta\|_\infty^2,
	\nn
	\end{align}
	so that $\bar I_f$ is {\Frechet} differentiable with derivative 
	\begin{align}
	[D \bar I_f(X)\, \delta]_s
	& = \int_t^s D f(X_\sigma)\, \delta_\sigma\, d\sigma
	\label{eq:D-I-f}
	\end{align}
	for all $X,\delta\in C([t,T];\cX^2)$, $s\in[t,T]$. So, recalling \er{eq:pre-D-X-of-Y}, and \er{eq:df},
	\begin{align}
	\ol{X}(Y+h)_s - \ol{X}(Y)_s - U_{s,t}(Y)\, h
	& = [d[{\bar I_f}]_{\ol{X}(Y)} ( \ol{X}(Y+h) - \ol{X}(Y) )]_s \, \| \ol{X}(Y+h) - \ol{X}(Y) \|_\infty
	\nn\\
	& \quad + \int_t^s Df(\ol{X}(Y)_\sigma) \, [ \ol{X}(Y+h)_\sigma - \ol{X}(Y)_\sigma - U_{\sigma,t}(Y)\, h ]\, d\sigma.
	\nn
	\end{align}
	Noting that $L\doteq \sup_{\sigma\in[t,T]} \|D f(\ol{X}(Y)_\sigma) \|_{\bo(\cX^2)} < \infty$, taking the norm of both sides yields
	\begin{align}
	\| \ol{X}(Y+h)_s - \ol{X}(Y)_s - U_{s,t}(Y)\, h \|
	& \le \| d[{\bar I_f}]_{\ol{X}(Y)} ( \ol{X}(Y+h) - \ol{X}(Y) ) \|_\infty \, \| \ol{X}(Y+h) - \ol{X}(Y) \|_\infty
	\nn\\
	& \quad + \int_t^s L\, \| \ol{X}(Y+h)_\sigma - \ol{X}(Y)_\sigma - U_{\sigma,t}(Y)\, h \| \, d\sigma.
	\nn
	\end{align}
	Hence, by Gronwall's inequality,
	\begin{align}
	& \hspace{-10mm} \| \ol{X}(Y+h)_s - \ol{X}(Y)_s - U_{s,t}(Y)\, h \|
	\nn\\
	& \le \| d[{\bar I_f}]_{\ol{X}(Y)} ( \ol{X}(Y+h) - \ol{X}(Y) ) \|_\infty \, \| \ol{X}(Y+h) - \ol{X}(Y) \|_\infty\,
	\exp(L\, (T-t)),
	\nn
	\end{align}
	or, with $\theta_Y(h) \doteq  \| d[{\bar I_f}]_{\ol{X}(Y)} ( \ol{X}(Y+h) - \ol{X}(Y) ) \|_\infty$,
	\begin{align}
	\| \ol{X}(Y+h) - \ol{X}(Y) - U_{\cdot,t}(Y)\, h \|_\infty
	& \le \theta_Y(h) \, \| \ol{X}(Y+h) - \ol{X}(Y) \|_\infty\, \exp(L\, (T-t))
	\nn\\
	& \le \theta_Y(h) \, \| \ol{X}(Y+h) - \ol{X}(Y) - U_{\cdot,t}(Y) \, h \|_\infty\, \exp(L\, (T-t)) 
	\nn\\
	& \qquad
	+ 
	\theta_Y(h) \, \sup_{s\in[t,T]} \| U_{s,t}(Y)\|_{\bo(\cX^2)} \, \| h \| \, \exp(L\, (T-t)).
	\nn
	\end{align}
	As $\theta_Y$ is continuous at $0$, there exists an $r>0$ sufficiently small such that
	$h\in\cB_0(r)$ implies that $\theta_Y(h) \exp(L\, (T-t)) < \demi$.
	Hence, with $h\in\cB_0(r)$,
	\begin{align}
	\| \ol{X}(Y+h) - \ol{X}(Y) - U_{\cdot,t}(Y)\, h \|_\infty
	& < 2\, \theta_Y(h) \, \sup_{s\in[t,T]} \| U_{s,t}(Y) \|_{\bo(\cX^2)} \, \| h \| \, \exp(L\, (T-t))
	\nn\\
	& = Q\  \theta_Y(h) \, \|h\|,
	\nn
	\end{align}
	in which $Q \doteq 2\, \sup_{s\in[t,T]} \| U_{s,t}(Y) \|_{\bo(\cX^2)}\, \exp(L\, (T-t))$. Consequently, taking a limit,
	\begin{align}
	& \lim_{\|h\|\rightarrow 0}
	\frac{\| \ol{X}(Y+h) - \ol{X}(Y) - U_{\cdot,t}(Y)\, h \|_\infty}{\|h\|}
	\le \lim_{\|h\|\rightarrow 0}  Q\, \theta_Y(h)  = 0.
	\nn
	\end{align}
	That is, $Y\mapsto\ol{X}(Y)$ is {\Frechet} differentiable, with the indicated derivative.
\end{proof}

\begin{proof}[Proof of Lemma \ref{lem:U-dble}]
	Fix $T\in\R_{>0}$, $t\in[0,T]$ as per the lemma statement. It is first demonstrated that $Y\mapsto U_{s,r}(Y)$ is continuous, uniformly in $r,s\in[t,T]$, as this motivates the subsequent proof of continuous differentiability.
	Fix $r,s\in[t,T]$, $h,\hat h\in\cX^2$. As $U_{s,r}(Y)\in\bo(\cX^2)$ is an element of the two-parameter family of evolution operators generated by $A(Y)_s\in\bo(\cX^2)$, see \er{eq:generator}, 
	\begin{align}
	U_{s,r}(Y)\, h & = h + \int_r^s A(Y)_\sigma\, U_{\sigma,r}(Y) \, h \, d\sigma,
	\nn\\
	U_{s,r}(Y+\hat h)\, h & = h + \int_r^s A(Y + \hat h)_\sigma\, U_{\sigma,r}(Y+\hat h) \, h \, d\sigma,
	\nn
	\end{align}
	so that 
	\begin{align}
	& [U_{s,r} (Y+ \hat h) - U_{s,r}(Y)]\, h
	= \int_r^s [ A(Y + \hat h)_\sigma\, U_{\sigma,r}(Y+\hat h) - A(Y)_\sigma\, U_{\sigma,r}(Y) ]\, h \, d\sigma
	\nn\\
	& = \int_r^s [ A(Y + \hat h)_\sigma - A(Y)_\sigma ]\,  [ U_{\sigma,r}(Y+\hat h) - U_{\sigma,r}(Y) ]\, h \, d\sigma
	\nn\\
	& \qquad 
	+ \int_r^s A(Y)_\sigma \, [ U_{\sigma,r}(Y+\hat h) - U_{\sigma,r}(Y) ]\, h\, d\sigma 
	+ \int_r^s [ A(Y + \hat h)_\sigma - A(Y)_\sigma ] \, U_{\sigma,r}(Y) \, h\, d\sigma.
	\label{eq:pre-Gronwall-U-cts-0}
	\end{align}
	Hence, by the triangle inequality,
	\begin{align}
	\| [ U_{s,r} (Y+ \hat h) - U_{s,r}(Y) ]\, h \|
	& \le \int_r^s \| A(Y + \hat h)_\sigma - A(Y)_\sigma \|_{\bo(\cX^2)} \, \| [ U_{\sigma,r}(Y+\hat h) - U_{\sigma,r}(Y) ]\, h\|\, d\sigma
	\nn\\
	& \qquad
	+ \int_r^s \| A(Y)_\sigma \|_{\bo(\cX^2)} \, \| [ U_{\sigma,r}(Y+\hat h) - U_{\sigma,r}(Y) ]\, h\|\, d\sigma
	\nn\\
	& 
	\qquad 
	+ \int_r^s \| A(Y + \hat h)_\sigma - A(Y)_\sigma \|_{\bo(\cX^2)} \, \| U_{\sigma,r}(Y) \, h \| \, d\sigma.
	\label{eq:pre-Gronwall-U-cts}
	\end{align}
	Recall \er{eq:ass-M-V}, and in particular the uniform bound on $x\mapsto D\grad^2 V(x)$. Given $x,\bar x\in\cX^2$, the mean value theorem implies that $\grad^2 V(x) - \grad^2 V(\bar x) = ( \int_0^1 D\grad^2 V(\bar x + \eta\, (x - \bar x)) \, d\eta ) (x - \bar x)$, so that $\| \grad^2 V(x) - \grad^2 V(\bar x)\|_{\bo(\cX)} \le \frac{K}{2}\, \|x - \bar x\|$ by \er{eq:ass-M-V}. 
	Hence, by \er{eq:generator}, there exists an $\alpha_1\in\R_{\ge 0}$ such that $\Lambda:\cX^2\rightarrow\bo(\cX^2)$ satisfies $\| \Lambda(Z) - \Lambda(\bar Z) \|_{\bo(\cX^2)} \le \alpha_1 \| Z - \bar Z \|$ for all $Z,\bar Z\in\cX^2$. So, applying Lemma \ref{lem:X-of-Y-cts}, there exists an $\alpha\in\R_{\ge 0}$, 
	$L_0 \doteq \sup_{\sigma\in[t,T]}\| A(0)_\sigma \|_{\bo(\cX^2)}$, $L_1 \doteq \alpha_1\, \exp(\alpha\, (T-t))$, 
	such that
	\begin{align}
	\sup_{\sigma\in[t,T]} \|A(Y+\hat h)_\sigma - A(Y)_\sigma \|_{\bo(\cX^2)}
	& \le \alpha_1 \sup_{\sigma\in[t,T]} \| \ol{X}(Y+\hat h)_\sigma - \ol{X}(Y)_\sigma \|
	\le L_1 \, \|\hat h\|,
	\label{eq:A-of-X-bar-cts}
	\\
	\sup_{\sigma\in[t,T]} \| A(Y)_\sigma \|_{\bo(\cX^2)} 
	& \le L_0 + L_1 \, \| \hat h \|,
	\nn
	\end{align}
	in which the second inequality follows from the first, via the triangle inequality, by selecting $\hat h = -Y$. Note further that as $\sigma\mapsto A(Y)_\sigma$ is continuous, $L_2 \doteq \sup_{\sigma\in [t,T]} \| U_{\sigma, t} (Y) \|_{\bo(\cX^2)} < \infty$, see \cite[Theorem 5.2, p.128]{P:83}.  Hence, substituting these inequalities in \er{eq:pre-Gronwall-U-cts} yields
	\begin{align}
	\| [ U_{s,r} (Y+ \hat h) - U_{s,r}(Y) ]\, h \|
	& \le (L_0 + 2\, L_1 \, \|\hat h\|) \int_r^s \| [ U_{\sigma,r}(Y+\hat h) - U_{\sigma,r}(Y) ]\, h\|\, d\sigma
	+ (T-t)\, L_1 \, L_2\, \|\hat h\| \, \|h\|.
	\nn
	\end{align}
	Gronwall's inequality subsequently implies that
	\begin{align}
	& \| [ U_{s,r} (Y+ \hat h) - U_{s,r}(Y) ]\, h \| 
	\le (T-t)\, L_1\, L_2\, \|\hat h\| \, \|h\|\, \exp( (L_0 + 2\, L_1 \, \|\hat h\|) (T-t))
	\nn\\
	& \Longrightarrow \quad
	\sup_{r,s\in[t,T]} \| U_{s,r} (Y+ \hat h) - U_{s,r}(Y)\|_{\bo(\cX^2)}
	\le (T-t)\, L_1\, L_2\, \|\hat h\| \, \exp( (L_0 + 2\, L_1 \, \|\hat h\|) (T-t)).
	\label{eq:U-diff-bound}
	\end{align}
	Continuity of $Y\mapsto U_{s,r}(Y)$, uniformly in $r,s\in[t,T]$, thus follows.
	
	Now, we show that $Y\mapsto U_{s,r}(Y)$ is {\Frechet} differentiable, uniformly in $r,s\in[t,T]$. Appealing to  the contraction theorem and Picard's principle, for any $t\leq r<s\leq T$ and $Y\in \cX$ we consider the two-parameter family of operators $V_{s,r}(y)\in \bo(\cX^2;\bo(\cX^2))$ solving
	\begin{align}
	V_{s,r}(Y)\, \hat h \, h 
	& = \int_r^s A(Y)_\sigma \, V_{\sigma,r}(Y)\, \hat h\, h\, d\sigma + 
	\int_r^s D_Y A(Y)_\sigma\, \hat h\, U_{\sigma,r}(Y)\, h\, d\sigma
	\label{eq:V-def}
	\end{align}
	for all $h, \hat h\in\cX^2$, $r,s\in[t,T]$, in which $D_Y A(Y)_\sigma = D \Lambda(\ol{X}(Y)_\sigma)\, U_{\sigma,t}(Y) \in \bo(\cX^2;\bo(\cX^2))$ by the chain rule and Lemma \ref{lem:X-of-Y-dble}. Note in particular by \er{eq:ass-M-V}, \er{eq:generator}, and Lemma \ref{lem:X-of-Y-cts} that 
	\begin{align}
	L_3 & \doteq \sup_{\sigma\in[t,T]} \| D \Lambda(\ol{X}(Y)_\sigma) \|_{\bo(\cX^2;\bo(\cX^2))} < \infty.
	\nn
	\end{align}
	Applying the triangle inequality to \er{eq:V-def}, and recalling the definitions of $L_0$, $L_1$, $L_2$, yields
	\begin{align}
	\| V_{s,r}(Y)\, \hat h \, h \|
	& \le \int_r^s (L_0 + L_1\, \|\hat h\|) \, \| V_{s,r}(Y)\, \hat h\, h\| \, d\sigma 
	+ \int_r^s L_3\, \|\hat h\| \, L_2\, \|h\|\, d\sigma 
	\nn\\
	& \le (T-t)\, L_2\, L_3\, \|\hat h\|\, \|h\| + (L_0 + L_1\, \|\hat h\|) \int_r^s \| V_{s,t}(Y)\, \hat h\, h\| \, d\sigma,
	\nn
	\end{align}
	so that by Gronwall's inequality, 
	\begin{align}
	\| V_{s,r}(Y)\, \hat h \, h \|
	& \le (T-t)\, L_2\, L_3\, \|\hat h\|\, \|h\|\, \exp\left( (L_0 + L_1\, \|\hat h\|)\, (T-t) \right).
	\nn
	\end{align}
	As $\hat h, h\in\cX^2$ are arbitrary, it follows immediately that $V_{s,r}(Y)\in\bo(\cX^2;\bo(\cX^2))$ for all $r,s\in[t,T]$. Recalling \er{eq:pre-Gronwall-U-cts-0}, observe by adding and subtracting terms that
	\begin{align}
	& [U_{s,r} (Y+ \hat h) - U_{s,r}(Y) - V_{s,r}(Y)\, \hat h]\, h
	\nn\\
	& 
	= \int_r^s [ A(Y + \hat h)_\sigma\, U_{\sigma,r}(Y+\hat h) - A(Y)_\sigma\, U_{\sigma,r}(Y) ]\, h \, d\sigma 
	- V_{s,r}(Y)\, \hat h\, h
	\nn\\
	& = \int_r^s A(Y)_\sigma \, [ U_{\sigma,r}(Y+\hat h) - U_{\sigma,r}(Y) - V_{\sigma,r}(Y)\, \hat h ]\, h\, d\sigma 
	\nn\\
	& \qquad 
	+ \int_r^s [ A(Y + \hat h)_\sigma - A(Y)_\sigma ]\,  [ U_{\sigma,r}(Y+\hat h) - U_{\sigma,r}(Y) ]\, h \, d\sigma
	\nn\\
	& \qquad
	+ \int_r^s [ A(Y + \hat h)_\sigma - A(Y)_\sigma - D_Y A(Y)_\sigma\, \hat h ] \, U_{\sigma,r}(Y) \, h\, d\sigma
	\nn\\
	& \qquad 
	- \left[ V_{s,r}(Y)\, \hat h\, h
	- \int_r^s A(Y)_\sigma\, V_{\sigma,r}(Y)\, \hat h \, h\, d\sigma 
	- \int_r^s  D_Y A(Y)_\sigma\, \hat h \, U_{\sigma,r}(Y) \, h\, d\sigma \right],
	\label{eq:pre-Gronwall-diff}
	\end{align}
	and the last term in square brackets is zero by definition \er{eq:V-def} of $V_{s,r}(Y)$. Define $\hat A:C([t,T];\cX^2)\rightarrow C([t,T];\bo(\cX^2))$ by $\hat A(X)_\sigma \doteq A(X_\sigma)$ for all $X\in C([t,T];\cX^2)$, and note that the range of $\hat A$ follows by \er{eq:ass-M-V}, \er{eq:generator}. Combining \er{eq:ass-M-V}, \er{eq:generator} with the mean value theorem, there exists $\hat\alpha\in\R_{\ge 0}$ such that
	\begin{align}
	& \|\hat A(X + \delta) - \hat A(X) - D A(X)\, \delta \|_{C([t,T];\bo(\cX^2))}
	= \sup_{\sigma\in[t,T]} \| A(X_\sigma + \delta_\sigma) - A(X_\sigma) - D A(X_\sigma)\, \delta_\sigma\|_{\bo(\cX^2)}
	\nn\\
	& = \sup_{\sigma\in[t,T]} \left\| \left( \ba{cc}
	0 & 0 \\
	\grad^2 V([X_\sigma+\delta_\sigma]_1) - \grad^2 V([X_\sigma]_1) - D\grad^2 V([X_\sigma]_1)\, [\delta_\sigma]_1
	& 0
	\ea \right) \right\|_{\bo(\cX^2)}
	\nn\\
	& \le \hat\alpha\, \sup_{\sigma\in[t,T]} \| \grad^2 V([X_\sigma+\delta_\sigma]_1) - \grad^2 V([X_\sigma]_1)
	- D\grad^2 V([X_\sigma]_1)\, [\delta_\sigma]_1 \|
	\nn\\
	& \le \hat\alpha\, \sup_{\sigma\in[t,T]} \left\| \left( 
	\int_0^1 D^2 \grad^2 V([X_\sigma]_1 + \eta [\delta_\sigma]_1) \, d\eta 
	\right) ([\delta_\sigma]_1, [\delta_\sigma]_1) \right\|_{\bo(\cX^2)}
	\nn\\
	& \le \hat\alpha\, \sup_{\sigma\in[t,T]} \int_0^1 \| D^2\grad^2 V([X_\sigma]_1 + \eta [\delta_\sigma]_1) \|_{\bo(\cX^2\times\cX^2;\bo(\cX^2))}\, d\eta\,
	\sup_{\sigma\in[t,T]}\|\delta_\sigma\|^2
	\nn\\
	& \le \ts{\frac{\hat\alpha\, K}{2}}\, \|\delta\|_\infty^2
	\label{eq:A-dble}
	\end{align}
	for all $X,\delta\in\cX^2$. Dividing both sides by $\|\delta\|_\infty$ and taking the limit as $\|\delta\|_\infty\rightarrow 0$ subsequently yields that $\hat A$ is {\Frechet} differentiable with derivative $D A(X)\in\bo(C([t,T];\cX^2);C([t,T];\bo(\cX^2))$. Hence, taking the norm of both sides of \er{eq:pre-Gronwall-diff}, applying the triangle inequality, \er{eq:U-diff-bound}, \er{eq:A-dble}, and recalling the definitions of $L_1$, $L_2$, $L_3$,
	\begin{align}
	& \| [U_{s,r} (Y+ \hat h) - U_{s,r}(Y) - V_{s,r}(Y)\, \hat h]\, h \|
	\nn\\
	& \le (L_0 + L_1\, \|\hat h\|) \int_r^s \| [U_{s,r} (Y+ \hat h) - U_{s,r}(Y) - V_{s,r}(Y)\, \hat h]\, h \| \, d\sigma
	\nn\\
	& \qquad
	+ (T-t)^2\, L_1^2\, L_2\, \|\hat h\|^2 \, \exp( (L_0 + 2\, L_1 \, \|\hat h\|) (T-t))\, \|h\|
	\nn\\
	& \qquad 
	+ (T-t) \, L_2\, \| \hat A\circ \ol{X}(Y+\hat h) - \hat A\circ \ol{X}(Y) 
	- D \hat A(\ol{X}(Y))\, D\ol{X}(Y)\, \hat h \|_{C([t,T];\bo(\cX^2))} \, \|h\|
	\nn\\
	& = (L_0 + L_1\, \|\hat h\|) \int_r^s \| [U_{s,r} (Y+ \hat h) - U_{s,r}(Y) - V_{s,r}(Y)\, \hat h]\, h \| \, d\sigma
	\nn\\
	& \qquad
	+ (T-t)^2\, L_1^2\, L_2\, \|\hat h\|^2 \, \exp( (L_0 + 2\, L_1 \, \|\hat h\|) (T-t))\, \|h\|
	\nn\\
	& \qquad
	+ (T-t)\, L_2\, \| d(\hat A\circ\ol{X})_{Y}(\hat h)\|_{C([t,T];\bo(\cX^2))}\, \|\hat h\|\, \|h\|,
	\nn
	\end{align}
	in which $d(\hat A\circ\ol{X})_Y(\cdot)$ is defined via \er{eq:df}. Hence, by Gronwall's inequality,
	\begin{align}
	&
	\| [U_{s,r} (Y+ \hat h) - U_{s,r}(Y) - V_{s,r}(Y)\, \hat h]\, h \|
	\nn\\
	& \le (T-t)\, L_2\, \left[ (T-t)\, L_1^2\, \|\hat h\| \, \exp( (L_0 + 2\, L_1 \, \|\hat h\|) (T-t))
	+ \| d(\hat A\circ\ol{X})_{Y}(\hat h)\|_{C([t,T];\bo(\cX^2))} \right] \|\hat h\| \, \|h\|
	\nn\\
	& \qquad
	\times \exp( (L_0 + L_1\, \|\hat h\|)\, (T-t)).
	\nn
	\end{align}
	As $\hat h, h\in\cX^2$ are arbitrary, 
	\begin{align}
	& \lim_{\|\hat h\|\rightarrow 0} \frac{\sup_{r,s\in[t,T]} \| U_{s,r} (Y+ \hat h) - U_{s,r}(Y) - V_{s,r}(Y)\, \hat h \|_{\bo(\cX^2)}}{\|\hat h\|}\\
	&\le \lim_{\|\hat h\|\rightarrow 0} \| d(\hat A\circ\ol{X})_{Y}(\hat h)\|_{C([t,T];\bo(\cX^2))} = 0.
	\nn
	\end{align}
	Hence, $Y\mapsto U_{s,r}(Y)$ is {\Frechet} differentiable, uniformly in $r,s\in[t,T]$, with derivative $V_{s,r}(Y)$. An analogous argument to \er{eq:pre-Gronwall-U-cts-0}, \er{eq:pre-Gronwall-U-cts} applied to \er{eq:V-def} can be applied to show that $Y\mapsto V_{s,r}(Y)$ is continuous, uniformly in $r,s\in[t,T]$, and the details are omitted.
\end{proof}
\begin{remark}
	We remark that, from the smoothness of $V$, we can compute the second order {\Frechet} derivative of $Y\mapsto U_{s,r}(Y)$. Indeed, using a first order expansion, we have that $A(z+\delta)=A(z)+DA(z)(\delta)+o(|\delta|)$, where $o(.)$ stands  for the second order integral rest. Now, fix $r>0$. In the same way as in (\ref{eq:pre-Gronwall-U-cts-0}), we have, putting $U_Y(s):=U_{s,r}(Y)$, for all $s\in[r,T]$
	\[ \dot{U}_{Y+\hat h}(s)-\dot U_{Y}(s)=A(Y)(s)(U_{Y+\hat h}(s)-U_{Y}(s))+DA(Y)(s)(\hat h)U_{Y+\hat h}(s) +o(|\hat h|)(s) U_{Y+\hat h}(s). \]
	So, letting $\xi(.)=U_{Y+\hat h}(.)-U_{Y}(.)$, $A(.)=A(Y)(.)$,  $Q(.)=DA(Y)(.)(\hat h)$, $v(.)=U_{Y+\hat h}(.)$, and $R(.)= o(|\hat h|)(.) U_{Y+\hat h}(.) $, the previous equation reduce to the ODE $\dot\xi=A\xi+Qv+R$, solvable via standard tools. Then, we have that $ \xi (s)=X(r) \hat h+\int_r^s X(s) X(t)^{-1}Q(t) v(t) \, dt+\int_r^s X(s) X(t)^{-1}R(t) \, dt$ where $X(.)$ is the fundamental solution of
	$ \dot X=A X  \text{ a.e. on }  [r,T],$ $X(r)=I.$ We have that $\int_r^s X(s) X(t)^{-1}R(t) \, dt$ is closed to zero for $\hat h\rightarrow 0$ (uniformly), and hence $X(r) \hat h+\int_r^s X(s) X(t)^{-1}Q(t) v(t)$ provide the derivative in (\ref{eq:V-def}). Moreover, it is possible to iterate such argument, in order to compute high order {\Frechet} derivatives of $Y\mapsto U_{s,r}(Y)$, as times as the potential $V$ is differentiable.
\end{remark}

\section{An auxiliary statement of Proposition \ref{prop:p-cost-dble}.}



\newcommand{\ddttwo}[2]		{{\frac{d^2 {#1}}{d {#2}^2}}}

\begin{proposition}
	Given $T\in\R_{>0}$, $t\in[0,T)$, $x,p\in\cX$, and $(\bar x_s, \bar p_s) \doteq \ol{X}(Y_p(x))_s$ for all $s\in[t,T]$, the maps $s\mapsto \grad_p \bar J_T(s,\bar x_s, \bar p_s)$ and $s\mapsto \grad_x \bar J_T(s,\bar x_s, \bar p_s)$ are continuously differentiable, with derivatives given by
	\begin{align}
	\ts{\ddtone{}{s}} \left[ \grad_p \bar J_T(s,\bar x_s, \bar p_s) \right]
	& = -\op{M}^{-1} \left( \bar p_s - \grad_x \bar J_T(s,\bar x_s,\bar p_s) \right),
	\label{eq:d-ds-grad-p-J-bar}
	\\
	\ts{\ddtone{}{s}} \left[ \grad_x \bar J_T(s,\bar x_s,\bar p_s) \right]
	& = \grad V(\bar x_s) - \grad^2 V(\bar x_s)\, \grad_p \bar J_T(s,\bar x_s, \bar p_s),
	\label{eq:d-ds-grad-x-J-bar}
	\end{align}  
	for all $s\in(t,T)$. Moreover, $s\mapsto \grad_p \bar J_T(s,\bar x_s, \bar p_s)$ is twice continuously differentiable, and satisfies
	\begin{align}
	0 & = \ts{\ddttwo{}{s}}  \left[ \grad_p \bar J_T(s,\bar x_s, \bar p_s) \right] 
	+ \op{M}^{-1}\, \grad^2 V(\bar x_s)\, \grad_p \bar J_T(s,\bar x_s, \bar p_s),
	\label{eq:second-order-ODE-grad-p-J-bar}
	\end{align}
	for all $s\in(t,T)$.
\end{proposition}
\begin{proof}
	Fix $T\in\R_{>0}$, $x,p\in\cX$, and let $(\bar x_s, \bar p_s)\in\cX^2$, $s\in[t,T]$, be as per the lemma statement. Fix $h\in\cX$.
	Applying Proposition \ref{prop:X-cost-dble}, $(s,x,p)\mapsto \bar J_T(s,x,p)$ is twice continuously differentiable, and the order of differentiation may be swapped. In particular,
	\begin{align}
	\ts{\ddtone{}{s}} \, [ D_p \, \bar J_T(s, \bar x_s, \bar p_s)\, h ]
	& = \ts{\pdtone{}{s}}\, D_p \, \bar J_T(s,\bar x_s, \bar p_s)\, h +
	D_x\, [D_p \, \bar J_T(s,\bar x_s,\bar p_s)\, h]\, \dot{\bar x}_s  + 
	D_p \, [D_p \, \bar J_T(s,\bar x_s, \bar p_s)\, h]\, \dot{\bar p}_s
	\nn\\
	& = \ts{\pdtone{}{s}}\, D_p \, \bar J_T(s,\bar x_s, \bar p_s)\, h +
	(D_x\, [D_p \, \bar J_T(s,\bar x_s,\bar p_s)] \, \dot{\bar x}_s) \, h  + 
	(D_p \, [D_p \, \bar J_T(s,\bar x_s, \bar p_s)]\, \dot{\bar p}_s) \, h
	\nn\\
	& = \left[ D_p \, \ts{\pdtone{}{s}}\, \bar J_T(s,\bar x_s, \bar p_s) +
	D_x\, D_p \, \bar J_T(s,\bar x_s,\bar p_s) \, \dot{\bar x}_s +
	D_p \, D_p \, \bar J_T(s,\bar x_s, \bar p_s)\, \dot{\bar p}_s \right] h.
	\label{eq:d-ds-Dp-J-0}
	\end{align}
	Meanwhile, $\bar J_T$ satisfies \er{eq:verify-J-bar} by Theorem \ref{thm:verify-J-bar}, i.e.
	\begin{align}
	0 
	& = -\ts{\pdtone{}{s}} \bar J_T(s,x,p) - \demi\, \langle p, \, \op{M}^{-1}\, p \rangle + V(x) 
	+ D_x \bar J_T(s,x,p)\, \op{M}^{-1}\, p - D_p \bar J_T(s,x,p)\, \grad V(x),
	\label{eq:apvgx-HJB}
	\end{align}
	for all $s\in(t,T)$, $x,p\in\cX$. Differentiating \er{eq:apvgx-HJB} with respect to $p$,
	\begin{align}
	0 & = - D_p [ \ts{\pdtone{}{s}} \bar J_T(s,x,p)]\, h - \langle \op{M}^{-1}\, p, \, h \rangle 
	+ D_p\, D_x\, \bar J_T(s,x,p)\, h\, \op{M}^{-1}\, p + D_x\, \bar J_T(s,x,p)\, \op{M}^{-1}\, h
	\nn\\
	& \qquad
	- D_p\, D_p\, \bar J_T(s,x,p)\, \grad V(x)\, h
	\nn\\
	& = - \langle \op{M}^{-1}\, (p - \grad_x \bar J_T(s,x,p)), \, h \rangle 
	\nn\\
	& \qquad -\left[ D_p [ \ts{\pdtone{}{s}} \bar J_T(s,x,p)]\, h - D_p\, D_x\, \bar J_T(s,x,p)\, h\, \op{M}^{-1}\, p
	+ D_p\, D_p\, \bar J_T(s,x,p)\, \grad V(x)\, h
	\right]
	\nn\\
	& = - \langle \op{M}^{-1}\, (p - \grad_x \bar J_T(s,x,p)), \, h \rangle 
	\nn\\
	& \qquad -\left[ D_p \, \ts{\pdtone{}{s}} \bar J_T(s,x,p) - D_x\, D_p\, \bar J_T(s,x,p)\, \op{M}^{-1}\, p
	+ D_p\, D_p\, \bar J_T(s,x,p)\, \grad V(x) \right] h.
	\nn
	\end{align}
	Evaluating along the trajectory $s\mapsto (\bar x_s,\bar p_s)$ yields
	\begin{align}
	& \left[ D_p \, \ts{\pdtone{}{s}} \bar J_T(s,\bar x_s, \bar p_s) 
	+ D_x\, D_p\, \bar J_T(s,\bar x_s, \bar p_s)\, \dot{\bar x}_s
	+ D_p\, D_p\, \bar J_T(s,\bar x_s, \bar p_s)\, \dot{\bar p}_s \right] h
	\nn\\
	& \hspace{80mm}
	= - \langle \op{M}^{-1}\, (\bar p_s - \grad_x \bar J_T(s,\bar x_s, \bar p_s)), \, h \rangle 
	\nn
	\end{align}
	Hence, substitution in \er{eq:d-ds-Dp-J-0} yields
	\begin{align}
	\langle \ts{\ddtone{}{s}} [ \grad_p \bar J_T(s,\bar x_s,\bar p_s) ], \, h \rangle
	=
	\ts{\ddtone{}{s}} \, [ D_p \, \bar J_T(s, \bar x_s, \bar p_s)\, h ]
	& = - \langle \op{M}^{-1}\, (\bar p_s - \grad_x \bar J_T(s,\bar x_s, \bar p_s)), \, h \rangle.
	\label{eq:d-ds-Dp-J}
	\end{align}
	Recalling that $h\in\cX$ is arbitrary immediately yields \er{eq:d-ds-grad-p-J-bar}.
	
	Similarly, for \er{eq:d-ds-grad-x-J-bar}, observe that
	\begin{align}
	\ts{\ddtone{}{s}} [ D_x \bar J_T(s,\bar x_s,\bar p_s)\, h]
	& = \ts{\pdtone{}{s}} D_x \bar J_T(s,\bar x_s,\bar p_s)\, h 
	+ D_x\, [ D_x \bar J_T(s,\bar x_s,\bar p_s)\, h]\, \dot{\bar x}_s 
	+ D_p\, [ D_x \bar J_T(s,\bar x_s,\bar p_s)\, h]\, \dot{\bar p}_s
	\nn\\
	& = \ts{\pdtone{}{s}} D_x \bar J_T(s,\bar x_s,\bar p_s)\, h
	+ (D_x\, [ D_x \bar J_T(s,\bar x_s,\bar p_s)] \, \dot{\bar x}_s)\, h
	+ (D_p\, [ D_x \bar J_T(s,\bar x_s,\bar p_s)]\, \dot{\bar p}_s)\, h
	\nn\\
	& = [ D_x\, \ts{\pdtone{}{s}} \bar J_T(s,\bar x_s,\bar p_s) 
	+ D_x\, D_x \bar J_T(s,\bar x_s,\bar p_s)\, \dot{\bar x}_s
	+ D_p\, D_x \bar J_T(s,\bar x_s,\bar p_s)\, \dot{\bar p}_s ]\, h.
	\label{eq:d-ds-Dx-J-0}
	\end{align}
	Meanwhile, differentiating \er{eq:apvgx-HJB} with respect to $x$,
	\begin{align}
	0 & = -D_x [ \ts{\pdtone{}{s}} \bar J_T(s,x,p) ]\, h + D_x V(x)\, h 
	+ D_x \, [D_x \bar J_T(s,x,p)\, \op{M}^{-1}\, p]\, h
	- D_x \, [ D_p \bar J_T(s,x,p)\, \grad V(x) ]\, h
	\nn\\
	& = - \left[ D_x [ \ts{\pdtone{}{s}} \bar J_T(s,x,p) ]\, h 
	+ D_x\, [D_x \bar J_T(s,x,p)\, h] \, (-\op{M}^{-1}\, p) 
	+ (D_x\, [D_p \bar J_T(s,x,p)]\, h)\, \grad V(x)
	\right]
	\nn\\
	& \qquad + D_x V(x)\, h - D_p \bar J_T(s,x,p)\, D_x \grad V(x)\, h
	\nn\\
	& = - \left[ D_x \, \ts{\pdtone{}{s}} \bar J_T(s,x,p) 
	+ D_x\, D_x \bar J_T(s,x,p)\, (-\op{M}^{-1}\, p) 
	+ D_p\, D_x \bar J_T(s,x,p)\, \grad V(x) \right] h
	\nn\\
	& \qquad + \left[ D_x V(x) - D_p \bar J_T(s,x,p)\, D_x \grad V(x) \right] h
	\nn
	\end{align}
	Evaluating along the trajectory $s\mapsto (\bar x_s,\bar p_s)$ yields
	\begin{align}
	& \left[ D_x \ts{\pdtone{}{s}} \bar J_T(s,\bar x_s, \bar p_s)
	+ D_x\, D_x \bar J_T(s,\bar x_s, \bar p_s)\, \dot{\bar x}_s) 
	+ D_p\, D_x \bar J_T(s,\bar x_s, \bar p_s)\, \dot{\bar p}_s \right] h
	\nn\\
	& = \langle \grad V(\bar x_s) - \grad^2 V(\bar x_s)\, \grad_p \bar J_T(s,\bar x_s,\bar p_s), \, h \rangle.
	\nn
	\end{align}
	Hence, substitution in \er{eq:d-ds-Dx-J-0} yields
	\begin{align}
	\langle \ts{\ddtone{}{s}} [ \grad_x \bar J_T(s,\bar x_s,\bar p_s)], \, h \rangle
	& = \langle \grad V(\bar x_s) - \grad^2 V(\bar x_s)\, \grad_p \bar J_T(s,\bar x_s,\bar p_s), \, h \rangle.
	\nn
	\end{align}
	Recalling that $h\in\cX$ is arbitrary immediately yields \er{eq:d-ds-grad-x-J-bar}.
	
	The remaining assertion regarding twice differentiability is immediate by inspection of \er{eq:d-ds-grad-p-J-bar}, \er{eq:d-ds-grad-x-J-bar}, with
	\begin{align}
	\ts{\ddttwo{}{s}} [ \grad_p \bar J_T(s,\bar x_s,\bar p_s) ]
	& = -\op{M}^{-1} \left( \dot{\bar p}_s - \ts{\ddttwo{}{s}} [ \grad_p \bar J_T(s,\bar x_s,\bar p_s) ] \right)
	\nn\\
	& = -\op{M}^{-1} \left( \grad V(\bar x_s) 
	- [ \grad V(\bar x_s) - \grad^2 V(\bar x_s) \, \grad_p \bar J_T(s,\bar x_s,\bar p_s)] \right)
	\nn\\
	& = -\op{M}^{-1} \, \grad V(\bar x_s) \, \grad^2 V(\bar x_s) \, \grad_p \bar J_T(s,\bar x_s,\bar p_s),
	\nn
	\end{align}
	as required.
\end{proof}

\newpage

\bibliographystyle{plain}
\bibliography{action}

\end{document}